\documentclass[a4paper,12pt]{article}
\usepackage{latexsym}
\usepackage{amsmath}
\usepackage{amssymb}
\usepackage{enumerate}
\usepackage{bm}
\usepackage{mathrsfs}

\usepackage{color}

\usepackage{theorem}
\newtheorem{theorem}{Theorem}[section]
\newtheorem{proposition}[theorem]{Proposition}
\newtheorem{lemma}[theorem]{Lemma}
\newtheorem{corollary}[theorem]{Corollary}

\theorembodyfont{\rmfamily}
\newtheorem{remark}[theorem]{Remark}
\newtheorem{example}[theorem]{Example}
\newtheorem{definition}[theorem]{Definition}
\newtheorem{proof}{\textmd{\textit{Proof.}}}

\makeatletter

\@addtoreset{equation}{section}
\makeatother

\newcommand{\qedd}{\hfill \square}
\newcommand{\ve}{\varepsilon}
\newcommand{\ez}{\epsilon}
\newcommand{\del}{\partial}
\newcommand{\lra}{\longrightarrow}
\newcommand{\e}{\mathrm{e}}
\newcommand{\td}{\mathrm{d}}

\newcommand{\N}{\ensuremath{\mathbb{N}}}

\newcommand{\R}{\ensuremath{\mathbb{R}}}

\newcommand{\fm}{\ensuremath{\mathfrak{m}}}
\newcommand{\bs}{\ensuremath{\mathbf{s}}}
\newcommand{\bK}{\ensuremath{\mathbf{K}}}

\newcommand{\sB}{\ensuremath{\mathsf{B}}}

\newcommand{\sI}{\ensuremath{\mathsf{I}}}
\newcommand{\sJ}{\ensuremath{\mathsf{J}}}
\newcommand{\sR}{\ensuremath{\mathsf{R}}}
\newcommand{\sT}{\ensuremath{\mathsf{T}}}

\newcommand{\LL}{\ensuremath{\mathscr{L}}}

\def\vol{\mathop{\mathrm{vol}}\nolimits}
\def\diam{\mathop{\mathrm{diam}}\nolimits}

\def\Cut{\mathop{\mathrm{Cut}}\nolimits}

\def\Ric{\mathop{\mathrm{Ric}}\nolimits}
\def\trace{\mathop{\mathrm{trace}}\nolimits}

\def\lf{\left}
\def\r{\right}

\newcommand{\Grad}{\bm{\nabla}}
\newcommand{\Lap}{\bm{\Delta}}

\newcommand{\wz}[1]{\widetilde{#1}}

\title{Comparison theorems on weighted Finsler manifolds and spacetimes with $\ez$-range}
\author{Yufeng LU\thanks{
Department of Mathematics, Osaka University, Osaka 560-0043, Japan
({\sf yufenglu.math@gmail.com}, {\sf s.ohta@math.sci.osaka-u.ac.jp})}
\and
Ettore MINGUZZI\thanks{
Dipartimento di Matematica e Informatica ``U. Dini'', Universit\`a degli Studi di Firenze,
Via S.~Marta 3, I-50139 Firenze, Italy ({\sf ettore.minguzzi@unifi.it})}
\and
Shin-ichi OHTA\footnotemark[1] \textsuperscript{,}\thanks{
RIKEN Center for Advanced Intelligence Project (AIP),
1-4-1 Nihonbashi, Tokyo 103-0027, Japan}}
\date{\today}
\begin{document}

\maketitle

\begin{abstract}
We establish the Bonnet--Myers theorem, Laplacian comparison theorem,
and Bishop--Gromov volume comparison theorem
for weighted Finsler manifolds as well as weighted Finsler spacetimes,
of weighted Ricci curvature bounded below by using the weight function.
These comparison theorems are formulated with $\ez$-range introduced in our previous paper,
that provides a natural viewpoint of interpolating
weighted Ricci curvature conditions of different effective dimensions.
Some of our results are new even for weighted Riemannian manifolds
and generalize comparison theorems of Wylie--Yeroshkin and Kuwae--Li.
\end{abstract}


\section{Introduction}

A weighted manifold is a pair given by a manifold,
equipped with some metric, and a weight function on it.
A fundamental example is a Riemannian manifold $(M,g)$
and a measure $\fm=\e^{-\psi} \,\vol_g$ on it,
where $\vol_g$ is the Riemannian volume measure
induced from the Riemannian metric $g$ and $\psi$ is a weight function on $M$.
This kind of \emph{weighted manifolds}, also called  \emph{manifolds with density},
naturally arise in the convergence theory of spaces
(when a sequence collapses to a lower dimensional space),
in the study of Ricci solitons (a weighted analogue of Einstein manifolds),
and in the {\em needle decomposition} (also called the {\em localization};
needles are weighted even when the original space is not).
We shall be interested in comparison geometry for these structures.

As for the nature of the metric on the manifold,
the Riemannian case was the first to be studied \cite{BE,Li},
and then generalizations to Finsler manifolds \cite{Oint},
Lorentzian manifolds \cite{Ca}, and Lorentz--Finsler manifolds \cite{LMO}, followed.

In comparison geometry and geometric analysis of these weighted manifolds,
the \emph{weighted Ricci curvature},
also called the \emph{Bakry--\'Emery--Ricci curvature} and attributed to \cite{BE},
plays a central role.
The weighted Ricci curvature $\Ric_N$ includes a real parameter $N$
sometimes called the \emph{effective dimension}.
For $N \in [\dim M,+\infty]$, $N$ indeed acts as an upper bound of the dimension
in the sense that, if $\Ric_N$ is bounded below by a real number $K$ (in a suitable sense),
then the weighted space enjoys various properties as it has
the Ricci curvature $\ge K$ and the dimension $\le N$.
In particular, $\Ric_{\infty}$ is useful for investigations of dimension-free estimates.
Gaussian spaces $(\R^n,\|\cdot\|,\e^{-\frac{K}{2}\|x\|^2} \,\td x)$, $K>0$,
are typical examples of spaces satisfying $\Ric_{\infty} \ge K$.
One of the recent milestones is that $\Ric_N \ge K$ is equivalent to
the \emph{curvature-dimension condition} CD$(K,N)$ \`a la Lott--Sturm--Villani
for weighted Riemannian (or Finsler) manifolds \cite{LV,Oint,StI,StII,Vi}.
Recently this characterization was generalized to
the (unweighted) Lorentzian situation by McCann \cite{Mc},
followed by a synthetic investigation on Lorentzian length spaces in \cite{CM}.

It is interesting that the parameter $N$ in $\Ric_N$ can be negative,
though it might appear strange if one sticks to the above interpretation
of $N$ as a bound on the dimension from above.
Some comparison theorems can be generalized to the case of $\Ric_N \ge K$
with $N \in (-\infty,0)$ or more generally $N \in (-\infty,1]$,
including the curvature-dimension condition \cite{Oneg,Oneedle},
isoperimetric inequality \cite{Mil}, splitting theorem \cite{Wy},
as well as singularity and splitting theorems in the Lorentzian context \cite{WW1,WW2}.
Then Wylie--Yeroshkin \cite{WY} introduced a different kind of curvature bound,
\begin{equation}\label{eq:WY}
\Ric_1 \ge K\e^{\frac{4}{1-\dim M}\psi}g
\end{equation}
on a weighted Riemannian manifold $(M,g,\psi)$,
where the lower bound is not constant but a function depending on the weight function $\psi$.
This curvature bound naturally arises from a projectively equivalent connection
to the Levi-Civita connection.
Moreover, the \emph{$\psi$-completeness} condition introduced in \cite{Wy},
\begin{equation}\label{eq:p-cplt}
\limsup_{l \to \infty} \inf_{\eta} \int_0^l \e^{\frac{2}{1-\dim M}\psi(\eta(t))} \,\td t=\infty,
\end{equation}
where $\eta$ runs over all unit speed minimal geodesics of length $l$ with the same initial point,
also motivates the study of \eqref{eq:WY}.
In \cite{WY} they established the Bonnet--Myers theorem, Laplacian comparison theorem
and Bishop--Gromov volume comparison theorem among others.
We remark that those comparison theorems do not have counterparts under $\Ric_1 \ge K>0$,
therefore the nonconstant bound \eqref{eq:WY} is essential.
We refer to \cite{Sa1} for the case of manifolds with boundary,
\cite{Sa2} for the curvature-dimension condition, and to
\cite{KW,Wy1} for related works on the \emph{weighted sectional curvature}.
In \cite{KL}, Kuwae--Li considered weighted Riemannian manifolds with
\begin{equation}\label{eq:KL}
\Ric_N \ge K\e^{\frac{4}{N-\dim M}\psi}g, \qquad N \in (-\infty,1],
\end{equation}
and generalized the comparison results in \cite{WY} to the case of $N \in (-\infty,1)$
together with some probabilistic applications.

In our previous paper \cite{LMO}, we introduced the notion of \emph{$\ez$-range}
and its associated completeness condition for spacetimes.
The aim of the present article is to establish comparison theorems with $\ez$-range
which enable us to interpolate the conditions $\Ric_N \ge K$ and \eqref{eq:WY}
and explain the reason why \eqref{eq:WY} and \eqref{eq:KL}
are admissible for those results in \cite{KL,WY}
while $\Ric_N \ge K$ with $N \in (-\infty,1] \cup \{+\infty\}$ is not.
Precisely, we showed in \cite{LMO} some singularity theorems for weighted Finsler spacetimes
under $\Ric_N \ge 0$ and the \emph{$\ez$-completeness} condition
\[ \int \e^{\frac{2(\ez-1)}{\dim M -1}\psi(\dot{\eta}(t))} \,\td t=\infty \]
inspired by \eqref{eq:p-cplt}, where $\ez$ is taken from the $\ez$-range
\begin{equation}\label{eq:erange}
\epsilon=0 \,\text{ for } N=1, \qquad
 \vert\epsilon\vert < \sqrt{\frac{N-1}{N-n}} \,\text{ for } N \neq 1,n, \qquad
 \ez \in \R \,\text{ for } N=n.
\end{equation}
(In order to avoid confusion, in this introduction we always set $\dim M=n$,
though $\dim M=n+1$ in \cite{LMO} (and Sections~\ref{sc:LFmfd}, \ref{sc:LFcomp} below)
as usual in Lorentzian geometry.)
Note that, on the one hand, $\ez=0$ corresponding to \cite{WY} is admissible for all $N$
and $\ez=(N-1)/(N-n)$ as in \cite{KL} is allowed for $N \le 1$.
On the other hand, $\ez=1$ corresponding to the constant bound $\Ric_N \ge K$
(and the usual geodesic completeness) is admissible only for $N \in [n,+\infty)$.

We generalize comparison theorems in \cite{KL,WY} under appropriate curvature bounds
including $\ez$.
For example, our \emph{Bonnet--Myers theorem} (Theorem~\ref{th:F-BM})
in the case of a weighted Riemannian manifold $(M,g,\psi)$ asserts that, if
\[ \Ric_N \ge K\e^{\frac{4(\ez -1)}{n-1}\psi} g, \qquad \e^{-\frac{2(\ez-1)}{n-1}\psi} \le b \]
for some $N \in (-\infty,1] \cup [n,+\infty]$, $\ez$ in the $\ez$-range \eqref{eq:erange} and $K,b>0$,
then the diameter of $M$ is bounded above by $b\pi/\sqrt{cK}$, where
\[ c=\frac{1}{n-1} \bigg( 1-\ez^2 \frac{N-n}{N-1} \bigg) >0. \]
This recovers the standard Bonnet--Myers theorem for $N \in [n,+\infty)$, $\ez=1$ and $b=1$
($c=1/(N-1)$), as well as the results in \cite{KL,WY} for $N \in (-\infty,1]$
and $\ez=(N-1)/(N-n)$ ($c=1/(n-N)$)
(see Remark~\ref{rm:b} for an alternative statement in terms of
a deformed distance structure without the bound $\e^{-\frac{2(\ez-1)}{n-1}\psi} \le b$ on $\psi$).

Besides the Bonnet--Myers theorem,
we also establish the \emph{Laplacian comparison theorem}
and \emph{Bishop--Gromov volume comparison theorem}
(in the latter the weight function $\psi$ is induced from a given measure $\fm$ on $M$),
in both weighted Finsler manifolds and weighted Finsler spacetimes.
We remark that those results for $\ez \neq (N-1)/(N-n)$ with $N<1$
or for $\ez \neq 1$ with $N \in [n,+\infty]$ are new even in the weighted Riemannian setting.
Furthermore, for the Bonnet--Myers and Laplacian comparison theorems on Finsler manifolds,
our results cover both the unweighted case \cite{BCS}
and the weighted case associated with measures \cite{Oint,OSheat};
this unification is not included in the literature.
As for future work,
it would be interesting to compare our comparison theorems
on weighted Finsler spacetimes with the recent synthetic investigations in \cite{CM,Mc}.
We refer to \cite{KS1,KS2,KS3} for some follow-up works on comparison geometry with $\ez$-range.

This article is divided into two parts.
The first part is devoted to weighted Finsler manifolds.
We recall  necessary concepts in Finsler geometry in Section~\ref{sc:Fmfd}
and develop the comparison theorems with $\ez$-range in Section~\ref{sc:Fcomp}.
The second part is devoted to weighted Finsler spacetimes.
In Section~\ref{sc:LFmfd} we review Lorentz--Finsler geometry,
causality theory and some analytic notions.
Finally, in Section~\ref{sc:LFcomp} we obtain the Lorentzian versions of the comparison theorems.

Although some arguments could be unified to a single framework,
we shall discuss the Finsler and Lorentz--Finsler cases rather separately
and present the proofs of comparison theorems in their each common languages,
for the sake of accessibility and hopefully motivating interactions between
Riemannian and Lorentzian geometries.

\section{Preliminaries for Finsler manifolds}\label{sc:Fmfd}

We first consider comparison theorems on weighted Finsler manifolds.
We refer to \cite{BCS,Obook,Shlec} for the basics of Finsler geometry
(we will follow the notations in \cite{Shlec}).
Throughout this and the next sections,
let $M$ be a connected $C^\infty$-manifold without boundary of dimension $n \ge 2$.

\subsection{Finsler manifolds}\label{ssc:Fmfd}

Given local coordinates $(x^i)_{i=1}^n$ on an open set $U \subset M$,
we will always use the fiber-wise linear coordinates
$(x^i,v^j)_{i,j=1}^n$ of $TU$ such that
\[ v=\sum_{j=1}^n v^j \frac{\del}{\del x^j}\Big|_x \in T_xM, \qquad x \in U. \]

\begin{definition}[Finsler structures]\label{df:Fstr}
We say that a nonnegative function $F:TM \lra [0,+\infty)$ is
a \emph{$C^{\infty}$-Finsler structure} of $M$ if the following three conditions hold:
\begin{enumerate}[(1)]
\item(\emph{Regularity})
$F$ is $C^{\infty}$ on $TM \setminus 0$,
where $0$ stands for the zero section;

\item(\emph{Positive $1$-homogeneity})
It holds $F(cv)=cF(v)$ for all $v \in TM$ and $c>0$;

\item(\emph{Strong convexity})
The $n \times n$ symmetric matrix
\begin{equation}\label{eq:gij}
\big( g_{ij}(v) \big)_{i,j=1}^n :=
 \bigg( \frac{1}{2}\frac{\del^2 [F^2]}{\del v^i \del v^j}(v) \bigg)_{i,j=1}^n
\end{equation}
is positive-definite for all $v \in TM \setminus 0$.
\end{enumerate}
We call such a pair $(M,F)$ a ($C^{\infty}$-)\emph{Finsler manifold}.
\end{definition}

In other words, $F$ provides a smooth Minkowski norm on each tangent space
which varies smoothly in horizontal directions as well.
If $F(-v)=F(v)$ for all $v \in TM$, then we say that $F$ is \emph{reversible}
or \emph{absolutely homogeneous}.

For $x,y \in M$, we define the (asymmetric) \emph{distance} from $x$ to $y$ by
\[ d(x,y):=\inf_{\eta} \int_0^1 F\big( \dot{\eta}(t) \big) \,\td t, \]
where $\eta:[0,1] \lra M$ runs over all $C^1$-curves such that $\eta(0)=x$ and $\eta(1)=y$.
Note that $d(y,x) \neq d(x,y)$ can happen since $F$ is only positively homogeneous.
A $C^{\infty}$-curve $\eta$ on $M$ is called a \emph{geodesic}
if it is locally minimizing and has a constant speed with respect to $d$,
similarly to Riemannian or metric geometry.
See \eqref{eq:geod} below for the precise geodesic equation.
For $v \in T_xM$, if there is a geodesic $\eta:[0,1] \lra M$
with $\dot{\eta}(0)=v$, then we define the \emph{exponential map}
by $\exp_x(v):=\eta(1)$.
We say that $(M,F)$ is \emph{forward complete} if the exponential
map is defined on the whole $TM$.
Then the \emph{Hopf--Rinow theorem} ensures that any pair of points
is connected by a minimal geodesic and that
every forward bounded closed set is compact (see \cite[Theorem~6.6.1]{BCS};
$A \subset M$ is said to be \emph{forward bounded}
if $\sup_{y \in A} d(x,y)<\infty$ for some (or, equivalently, for all) $x \in M$).

For $v \in T_xM \setminus \{0\}$, the positive-definite matrix
$(g_{ij}(v))_{i,j=1}^n$ in \eqref{eq:gij} induces the Riemannian structure $g_v$ of $T_xM$ by
\begin{equation}\label{eq:gv}
g_v\Bigg( \sum_{i=1}^n a_i \frac{\del}{\del x^i}\Big|_x,
 \sum_{j=1}^n b_j \frac{\del}{\del x^j}\Big|_x \Bigg)
 := \sum_{i,j=1}^n g_{ij}(v) a_i b_j.
\end{equation}
Note that this definition is coordinate-free and $g_v(v,v)=F^2(v)$ holds.
One can regard $g_v$ as the best Riemannian approximation of $F|_{T_xM}$
in the direction $v$.
The \emph{Cartan tensor}
\[ C_{ijk}(v):=\frac{1}{2} \frac{\del g_{ij}}{\del v^k}(v),
 \qquad v \in TM \setminus 0, \]
measures the variation of $g_v$ in the vertical directions,
and vanishes everywhere on $TM \setminus 0$
if and only if $F$ comes from a Riemannian metric.
We remark that
\begin{equation}\label{eq:A}
\sum_{i=1}^n C_{ijk}(v) v^i =\sum_{j=1}^n C_{ijk}(v) v^j =\sum_{k=1}^n C_{ijk}(v) v^k =0
\end{equation}
by Euler's homogeneous function theorem (\cite[Theorem~1.2.1]{BCS}).

Define the \emph{formal Christoffel symbol}
\[ \gamma^i_{jk}(v):=\frac{1}{2}\sum_{l=1}^n g^{il}(v) \bigg\{
 \frac{\del g_{lk}}{\del x^j}(v) +\frac{\del g_{jl}}{\del x^k}(v)
 -\frac{\del g_{jk}}{\del x^l}(v) \bigg\} \]
for $v \in TM \setminus 0$,
where $(g^{ij}(v))$ denotes the inverse matrix of $(g_{ij}(v))$,
and the \emph{geodesic spray coefficients} and the \emph{nonlinear connection}
\[ G^i(v):=\frac{1}{2}\sum_{j,k=1}^n \gamma^i_{jk}(v) v^j v^k, \qquad
 N^i_j(v):=\frac{\del G^i}{\del v^j}(v) \]
for $v \in TM \setminus 0$ ($G^i(0)=N^i_j(0):=0$ by convention).
Observe that $G^i$ is positively $2$-homogeneous ($G^i(cv)=c^2 G^i(v)$ for $c>0$)
and we have $\sum_{j=1}^n N^i_j(v) v^j=2G^i(v)$.
By using $N^i_j$,
the coefficients of the \emph{Chern connection} are given by
\[ \Gamma^i_{jk}(v):=\gamma^i_{jk}(v)
 -\sum_{l,m=1}^n g^{il}(v) (C_{lkm}N^m_j +C_{jlm}N^m_k -C_{jkm}N^m_l)(v) \]
on $TM \setminus 0$.

\begin{definition}[Covariant derivative]\label{df:covd}
The \emph{covariant derivative} of a vector field $X$ by $v \in T_xM$
with \emph{reference vector} $w \in T_xM \setminus \{0\}$ is defined as
\[ D_v^w X(x):=\sum_{i,j=1}^n \Bigg\{ v^j \frac{\del X^i}{\del x^j}(x)
 +\sum_{k=1}^n \Gamma^i_{jk}(w) v^j X^k(x) \Bigg\} \frac{\del}{\del x^i}\Big|_x \in T_xM. \]
\end{definition}

The \emph{geodesic equation} is then written  with the help of \eqref{eq:A} as
\begin{equation}\label{eq:geod}
D_{\dot{\eta}}^{\dot{\eta}} \dot{\eta}(t)
 =\sum_{i=1}^n \big\{ \ddot{\eta}^i(t) +2G^i\big( \dot{\eta}(t) \big) \big\}
 \frac{\del}{\del x^i} \Big|_{\eta(t)} =0.
\end{equation}

\subsection{Jacobi fields and Ricci curvature}\label{ssc:Fcurv}

A $C^{\infty}$-vector field $J$ along a geodesic $\eta$ is called a \emph{Jacobi field}
if it is realized as the variational vector field of a variation consisting of geodesics,
namely $J(t)=\del \zeta/\del s(t,0)$ for some $\zeta:[0,l] \times (-\ve,\ve) \lra M$
such that $\zeta(t,0)=\eta(t)$ and $\zeta(\cdot,s)$ is geodesic for every $s \in (-\ve,\ve)$.
A Jacobi field is equivalently characterized by the equation
\[ D^{\dot{\eta}}_{\dot{\eta}} D^{\dot{\eta}}_{\dot{\eta}} J +R_{\dot{\eta}}(J) =0, \]
where
\[ R_v(w):=\sum_{i,j=1}^n R^i_j(v) w^j \frac{\del}{\del x^i}\Big|_x \]
for $v,w \in T_xM$ and
\[ R^i_j(v) :=2\frac{\del G^i}{\del x^j}(v)
 -\sum_{k=1}^n \bigg\{ \frac{\del N^i_j}{\del x^k}(v) v^k
 -2\frac{\del N^i_j}{\del v^k}(v) G^k(v) \bigg\}
 -\sum_{k=1}^n N^i_k(v) N^k_j(v) \]
is the \emph{curvature tensor}.

\begin{definition}[Curvatures]\label{df:F-curv}
For linearly independent tangent vectors $v,w \in T_xM$,
we define the \emph{flag curvature} by
\[ \bK(v,w) :=\frac{g_v(R_v(w),w)}{F^2(v) g_v(w,w) -g_v(v,w)^2}. \]
We then define the \emph{Ricci curvature} of $v$ by
\[ \Ric(v) :=F^2(v) \sum_{i=1}^{n-1} \bK(v,e_i), \]
where $\{e_i\}_{i=1}^{n-1} \cup \{v/F(v)\}$ is an orthonormal basis of $(T_xM,g_v)$,
and $\Ric(0):=0$.
\end{definition}

\begin{remark}\label{rm:curv}
Although we will not use it, here we explain a useful connection
between the Riemannian and Finsler curvatures (see, e.g., \cite{Au,Obook,Shlec}).
Given a nonzero vector $v \in T_xM$, let us extend it to a $C^{\infty}$-vector field $V$
on a neighborhood of $x$ such that every integral curve of $V$ is geodesic.
Then the Finsler flag curvature $\bK(v,w)$ for any $w$ coincides with
the sectional curvature of the plane spanned by $v$ and $w$
with respect to the Riemannian metric $g_V$.
In particular, the Finsler Ricci curvature $\Ric(v)$ coincides with
the Riemannian Ricci curvature $\Ric(v,v)$ with respect to $g_V$.
The condition that all integral curves are geodesic is essential.
This characterization sometimes enables us to reduce a Finsler problem to a Riemannian one.
\end{remark}

\subsection{Unweighted Laplacian}\label{ssc:F-Lap}

In order to introduce some analytic tools including the Laplacian and Hessian,
we need the \emph{dual Finsler structure} $F^*:T^*M \lra [0,+\infty)$ to $F$ defined by
\[ F^*(\omega) :=\sup_{v \in T_xM,\, F(v) \le 1} \omega(v)
 =\sup_{v \in T_xM,\, F(v)=1} \omega(v) \]
for $\omega \in T_x^*M$.
It is clear by definition that $\omega(v) \le F^*(\omega) F(v)$ holds.
In the coordinates $(x^i,\omega_j)_{i,j=1}^n$ of $T^*U$ given by
$\omega=\sum_{j=1}^n \omega_j \,\td x^j$,
we will also consider
\[ g^*_{ij}(\omega) :=\frac{1}{2} \frac{\del^2[(F^*)^2]}{\del \omega_i \del \omega_j}(\omega),
 \qquad i,j=1,2,\ldots,n, \]
for $\omega \in T^*U \setminus 0$.

Let us denote by $\LL^*:T^*M \lra TM$ the \emph{Legendre transform}.
Precisely, $\LL^*$ sends $\omega \in T_x^*M$ to the unique element $v \in T_xM$
such that $F(v)=F^*(\omega)$ and $\omega(v)=F^*(\omega)^2$.
In coordinates we can write down
\[ \LL^*(\omega)=\sum_{i,j=1}^n g_{ij}^*(\omega) \omega_i \frac{\del}{\del x^j}\Big|_x
 =\sum_{j=1}^n \frac{1}{2} \frac{\del[(F^*)^2]}{\del \omega_j}(\omega) \frac{\del}{\del x^j}\Big|_x \]
for $\omega \in T_x^*M \setminus \{0\}$ (the latter expression makes sense also at $0$).
Note that $g^*_{ij}(\omega) =g^{ij}(\LL^*(\omega))$ for $\omega \in T_x^*M \setminus \{0\}$.
The map $\LL^*|_{T^*_xM}$ is linear only when $F|_{T_xM}$
comes from an inner product.

For a $C^1$-function $f:M \lra \R$, we define the \emph{gradient vector field} of $f$ by
\[ \Grad f :=\LL^*(\td f) =\sum_{i,j=1}^n g^*_{ij}(\td f) \frac{\del f}{\del x^i} \frac{\del}{\del x^j}. \]
We remark that, to be precise, the latter expression makes sense provided $\td f \neq 0$.
If $f$ is $C^2$ and $\td f(x) \neq 0$, then we define the \emph{Hessian}
$\Grad^2 f:T_xM \lra T_xM$ of $f$ at $x$ by
\begin{equation}\label{eq:F-Hess}
\Grad^2 f(v) :=D^{\Grad f}_v(\Grad f).
\end{equation}
The Hessian is symmetric in the sense that
\[ g_{\Grad f}\big( \Grad^2 f(v),w \big) =g_{\Grad f}\big( v,\Grad^2 f(w) \big) \]
for all $v,w \in T_xM$ (see \cite[Lemma~2.3]{OSbw} or Lemma~\ref{lm:Hess} below).
Then we define the \emph{unweighted Laplacian} of a $C^2$-function $f:M \lra \R$ by
\begin{equation}\label{eq:F-Lap}
\Lap f :=\trace (\Grad^2 f)
\end{equation}
on $\{ x \in M \,|\, \td f(x) \neq 0 \}$.

When $(M,F)$ is equipped with a measure (as in Subsection~\ref{ssc:F-BG}),
we employ the \emph{weighted Laplacian} defined as the divergence
(associated with the measure) of the gradient vector field; see \cite{OSheat} for details.
In this article (except for Subsection~\ref{ssc:F-BG}),
more generally, we shall consider a weight function not necessarily induced from a measure.
Introducing a measure is necessary when we develop analysis on Finsler manifolds,
however, we remark that there is in general no canonical measure
on a Finsler manifold as good as the Riemannian volume measure
(see \cite{ORand} for a related discussion).

\section{Comparison theorems on weighted Finsler manifolds}\label{sc:Fcomp}

\subsection{Weighted Finsler manifolds}\label{ssc:wFmfd}

As a weight, following \cite{LMO},
we employ a positively $0$-homogeneous $C^{\infty}$-function
on the slit tangent bundle:
\[ \psi: TM \setminus 0 \lra \R, \qquad \psi(cv)=\psi(v) \,\ \text{for all}\ c>0. \]
For a nonconstant geodesic $\eta$, we define
\begin{equation}\label{eq:psi_eta}
\psi_{\eta}(t):=\psi \big( \dot{\eta}(t) \big).
\end{equation}

\begin{definition}[Weighted Ricci curvature]\label{df:F-wRic}
Given $v \in TM \setminus 0$, let $\eta:(-\ve,\ve) \lra M$ be the geodesic with $\dot{\eta}(0)=v$.
Then, for $N \in \R \setminus \{n\}$, define the \emph{weighted Ricci curvature} by
\begin{equation}\label{eq:F-wRic}
\Ric_N(v) :=\Ric(v) +\psi''_{\eta}(0) -\frac{\psi'_{\eta}(0)^2}{N-n}.
\end{equation}
We also define
\[ \Ric_{\infty}(v) :=\lim_{N \to \infty} \Ric_N(v) =\Ric(v) +\psi''_{\eta}(0),
 \qquad \Ric_n(v) :=\lim_{N \downarrow n} \Ric_N(v), \]
and $\Ric_N(0):=0$.
\end{definition}

By definition we observe the following monotonicity:
For $N \in (n,+\infty)$ and $N' \in (-\infty,1)$,
\begin{equation}\label{eq:mono}
\Ric_n(v) \le \Ric_N(v) \le \Ric_{\infty}(v) \le \Ric_{N'}(v) \le \Ric_1(v).
\end{equation}
Thereby bounding $\Ric_1$ from below is a weaker condition than that for $\Ric_{\infty}$.
By $\Ric_N \ge K$ we will mean that $\Ric_N(v) \ge KF^2(v)$ holds for some $K \in \R$
and all $v \in TM$.

This framework generalizes the weighted Ricci curvature
associated with a measure introduced in \cite{Oint} (see also \cite{Obook}).
When $M$ is equipped with a positive $C^{\infty}$-measure $\fm$
(i.e., in each local chart, the density function of $\fm$
with respect to the Lebesgue measure is positive and $C^{\infty}$),
the corresponding weight function $\psi_{\fm}$ is given by
\begin{equation}\label{eq:psi_m}
\td\fm =\e^{-\psi_{\fm}(\dot{\eta}(t))} \sqrt{\det\big[ g_{ij}\big(\dot{\eta}(t) \big) \big]}
 \,\td x^1 \td x^2 \cdots \td x^n
\end{equation}
along geodesics $\eta$.
Notice that $\sqrt{\det[g_{ij}(\dot{\eta}(t))]} \,\td x^1 \td x^2 \cdots \td x^n$ is the volume measure
for the Riemannian metric $g_{\dot{\eta}}$ along $\eta$.
Then, for $(M,F,\fm)$ satisfying $\Ric_N \ge K$, we can obtain various comparison theorems
including those we will extend in this article (\cite{Oint,OSheat}),
as well as the curvature-dimension condition (\cite{Oint,Oneg,Oneedle})
and the needle decomposition (\cite{Oneedle}) among others.
Compared with $\psi_{\fm}$,
our general weight function $\psi$ on $TM \setminus 0$
allows us to include in the analysis the \emph{unweighted case},
which is indeed recovered for $\psi \equiv 0$ (cf.\ \eqref{eq:F-wRic}).
We also remark that, in the Riemannian case, it is common to employ a function on $M$
as a weight function.
This is because any measure $\fm$ is written as $\fm=\e^{-\psi} \,\vol_g$,
and then $\psi \in C^{\infty}(M)$ is the weight function.

In our previous paper \cite{LMO}, inspired by Wylie's work \cite{Wy},
we introduced a completeness condition with respect to a parameter $\ez \in \R$
in a certain range specified later.
We shall follow the same lines in the Finsler setting.

\begin{definition}[$\ez$-completeness]\label{df:e-cplt}
A geodesic $\eta:[0,l) \lra M$ ($l \in (0,+\infty]$) is said to be \emph{forward $\ez$-complete} if
\[ \int_0^l \e^{\frac{2(\ez -1)}{n-1}\psi_{\eta}(t)} \,\td t =\infty. \]
We say that $(M,F,\psi)$ is \emph{forward $\ez$-complete}
if any geodesic $\eta:[0,\delta) \lra M$ in $M$
can be extended to a forward $\ez$-complete geodesic.
\end{definition}

The case of $\ez=1$ is the usual forward completeness in Finsler geometry,
and the case of $\ez=0$ was introduced in \cite{Wy}
and further studied in \cite{Sa1,Sa2,WY} for Riemannian manifolds.
We also remark that, if $(\ez-1)\psi$ is bounded below,
then the forward completeness implies the forward $\ez$-completeness.
The reason behind these different choices of $\ez$ is understood
by introducing the admissible range of $\ez$ depending on $N$,
called the \emph{$\ez$-range} introduced in \cite[Proposition~5.8]{LMO},
where we showed the existence of a conjugate point within the $\ez$-range.
In the current setting, we define as follows.

\begin{definition}[$\ez$-range]\label{df:F-eran}
Given $N \in (-\infty,1] \cup [n,+\infty]$, we will consider $\ez \in \R$ in the following
\emph{$\ez$-range}:
\begin{equation}\label{eq:F-eran}
\epsilon=0 \,\text{ for } N=1, \qquad
 \vert\epsilon\vert < \sqrt{\frac{N-1}{N-n}} \,\text{ for } N \neq 1,n, \qquad
 \ez \in \R \,\text{ for } N=n.
\end{equation}
We also define the associated constant $c =c(N,\ez)$ by
\begin{equation}\label{eq:F-c}
c :=\frac{1}{n-1}\lf(1-\ez^2\frac{N-n}{N-1}\r) >0
\end{equation}
for $N \ne 1$.
If $\ez=0$, then one can take $N \to 1$ and set $c(1,0):=1/(n-1)$.
\end{definition}

Note that $\ez=1$ is admissible only for $N \in [n,+\infty)$, while $\ez=0$ is always admissible.

\subsection{Bonnet--Myers theorem}\label{ssc:F-BM}

We first consider the Bonnet--Myers diameter bound taking the $\ez$-range into account.
The case of $N \in [n,+\infty)$ and $\ez=1$ (so that $c=1/(N-1)$) can be found in \cite{Oint}.

Let us first illustrate some common notations used in the proofs of the comparison theorems.
Given a unit tangent vector $v \in U_xM :=T_xM \cap F^{-1}(1)$,
let $\eta:[0,l) \lra \R$ be the geodesic with $\dot{\eta}(0)=v$.
We take an orthonormal basis $\{e_i\}_{i=1}^n$ of $(T_xM,g_v)$ with $e_n=v$
and consider the Jacobi fields
\[ E_i(t):=(\td\exp_x)_{tv}(te_i), \quad i=1,2,\ldots,n-1, \]
along $\eta$.
Define the $(n-1) \times (n-1)$ matrices $A(t)=(a_{ij}(t))$ and $B(t)=(b_{ij}(t))$ by
\[ a_{ij}(t):=g_{\dot{\eta}}\big( E_i(t),E_j(t) \big) ,\qquad
 D^{\dot{\eta}}_{\dot{\eta}}E_i(t) =\sum_{j=1}^{n-1} b_{ij}(t)E_j(t). \]
We also define $R(t)=(R_{ij}(t))$ by
\[ R_{ij}(t) :=g_{\dot{\eta}}\big( R_{\dot{\eta}}(E_i(t)),E_j(t) \big)
 =g_{\dot{\eta}}\big( R_{\dot{\eta}}(E_j(t)),E_i(t) \big). \]
We summarize some necessary properties of $A,B$ and $R$.

\begin{lemma}\label{lm:F-Bish}
\begin{enumerate}[{\rm (i)}]
\item We have $BA=AB^{\sT}$ and $A'=2BA$, where $B^{\sT}$ is the transpose of $B$.
\item $A^{-1/2}BA^{1/2}$ is symmetric.
\item The \emph{Riccati equation}
\begin{equation}\label{eq:F-Ricc}
A'' -2B^2 A +2R =0
\end{equation}
holds.
\end{enumerate}
\end{lemma}

See \cite[\S 7]{Oint} (or \cite[\S 8.1]{Obook}) for the proof of the lemma, here we only remark that
(ii) readily follows from $BA=AB^{\sT}$ in (i).
We shall prove the \emph{Bishop inequality} in the current setting,
inspired by \cite[Proposition~5.14]{LMO} for weighted Lorentz--Finsler manifolds
(see \cite[\S III.4]{Ch} for the Riemannian case).
This is an essential ingredient of all the comparison theorems in this section.

\begin{proposition}[Bishop inequality]\label{pr:F-Bish}
Let $v \in U_xM$, $\eta:[0,l) \lra M$, $A(t)$, $B(t)$ and $R(t)$ as above.
Given $N \in (-\infty,1] \cup [n,+\infty]$, $\ez$ in the $\ez$-range \eqref{eq:F-eran}
and $c=c(N,\ez)$ as in \eqref{eq:F-c}, we define
\[ h(t) := \e^{-c\psi_\eta(t)}\big(\! \det A(t) \big)^{c/2}, \qquad
 h_1(\tau):=h \big( \varphi_\eta^{-1}(\tau) \big) \]
for $t \in [0,l)$ and $\tau \in [0,\varphi_{\eta}(l))$, where
\begin{equation}\label{eq:phi_e}
\varphi_{\eta}(t) :=\int_0^t \e^{\frac{2(\ez -1)}{n-1}\psi_{\eta}(s)} \,\td s.
\end{equation}
Then, for all $\tau \in (0,\varphi_\eta(l))$, we have
\begin{equation}\label{eq:F-Bish}
h_1''(\tau) \le -ch_1(\tau)\Ric_N \big( (\eta\circ\varphi_\eta^{-1} \big)' (\tau) \big).
\end{equation}
\end{proposition}

When $N \in [n,+\infty)$ and $\ez=1$, we have $c=1/(N-1)$, $\varphi_{\eta}(t)=t$ and $h_1=h$.
Hence \eqref{eq:F-Bish} reduces to the Bishop inequality in the standard form:
\[ h''(t) \le -\frac{\Ric_N(\dot{\eta}(t))}{N-1}h(t). \]
Note also that the parametrization \eqref{eq:phi_e} has the same form as the $\ez$-completeness
(Definition~\ref{df:e-cplt}).

We give here a rather algebraic but streamlined proof.
A different proof, that might give further insights,
could be obtained along the lines of the analogous statement in Subsection~\ref{ssc:LF-BM}
for the Lorentz--Finsler case; see \eqref{eq:LF-Bish}.
That line of proof, however, would require more work in terms of preliminary definitions and results.

\begin{proof}
Put $h_0(t):=(\det A(t))^{1/(2(n-1))}$ and observe from Lemma~\ref{lm:F-Bish} that
\begin{align*}
(n-1)h'_0 &=\frac{h_0}{2}(\det A)^{-1} (\det A)' =\frac{h_0}{2}\trace(A' A^{-1}) =h_0 \trace(B), \\
(n-1)h''_0 &= h'_0 \trace(B) +\frac{h_0}{2}\trace\big( A'' A^{-1} -(A' A^{-1})^2 \big) \\
 &= \frac{h_0}{n-1} (\trace(B))^2 -h_0 \trace(RA^{-1}) -h_0 \trace(B^2).
\end{align*}
The Cauchy--Schwarz inequality (applied to the eigenvalues of $B$)
yields $( \trace B)^2 \le (n-1) \trace(B^2)$
(since $A^{-1/2} B A^{1/2}$ is symmetric),
and hence we obtain the unweighted Bishop inequality:
\begin{equation}\label{eq:L-Bi0}
h''_0(t) \le -\frac{\Ric(\dot{\eta}(t))}{n-1} h_0(t).
\end{equation}
This is the starting point of our estimate.

We first assume $N \in (-\infty,1) \cup (n,+\infty]$.
Since $h(t)=\e^{-c\psi_{\eta}(t)} h_0(t)^{c(n-1)}$, we have
\[ h' = h \cdot \bigg( c(n-1)\frac{h'_0}{h_0} -c\psi'_{\eta} \bigg) \]
and
\begin{align*}
&h'' = h \bigg( c(n-1)\frac{h'_0}{h_0} -c\psi'_{\eta} \bigg)^2
 +h \bigg\{ c(n-1)\frac{h_0 h''_0 -(h'_0)^2}{h_0^2} -c\psi''_{\eta} \bigg\} \\
&= ch \bigg\{ (n-1)\frac{h''_0}{h_0} -\psi''_{\eta} +\big( c(n-1)^2 -(n-1) \big) \frac{(h'_0)^2}{h_0^2}
 -2c(n-1)\frac{h'_0}{h_0}\psi'_{\eta} +c(\psi'_{\eta})^2 \bigg\} \\
&\le -ch \Ric_N(\dot{\eta}) \\
&\quad +ch \bigg\{ (n-1)\big( c(n-1)-1 \big) \frac{(h'_0)^2}{h_0^2}
 -2c(n-1)\frac{h'_0}{h_0}\psi'_{\eta} +\bigg( c-\frac{1}{N-n} \bigg) (\psi'_{\eta})^2 \bigg\},
\end{align*}
where we used \eqref{eq:L-Bi0}.
In order to estimate the remaining terms in the last line,
we observe from $h(t)=h_1(\varphi_{\eta}(t))$ that
\[ h' = h'_1(\varphi_{\eta}) \e^{\frac{2(\ez -1)}{n-1}\psi_{\eta}}, \qquad
 h'' = h''_1(\varphi_{\eta}) \e^{\frac{4(\ez -1)}{n-1}\psi_{\eta}}
 +h' \frac{2(\ez -1)}{n-1} \psi'_{\eta}. \]
Hence we have
\[ h''_1(\varphi_{\eta}) \e^{\frac{4(\ez -1)}{n-1}\psi_{\eta}}
 = h'' -ch \frac{2(\ez -1)}{n-1} \bigg( (n-1)\frac{h'_0}{h_0}\psi'_{\eta} -(\psi'_{\eta})^2 \bigg)
 \le -ch \Ric_N(\dot{\eta}) +ch \Phi, \]
where
\begin{align*}
\Phi &:= (n-1)\big( c(n-1)-1 \big) \frac{(h'_0)^2}{h_0^2}
 -2c(n-1)\frac{h'_0}{h_0}\psi'_{\eta} +\bigg( c-\frac{1}{N-n} \bigg) (\psi'_{\eta})^2 \\
&\ \quad -\frac{2(\ez -1)}{n-1} \bigg( (n-1)\frac{h'_0}{h_0}\psi'_{\eta} -(\psi'_{\eta})^2 \bigg).
\end{align*}
By substituting $c$ from \eqref{eq:F-c} and noticing $(N-n)/(N-1)>0$, we deduce that
\begin{align*}
\Phi &= -\ez^2\frac{(n-1)(N-n)}{N-1} \frac{(h'_0)^2}{h_0^2}
 -2\bigg( \ez -\ez^2 \frac{N-n}{N-1} \bigg) \frac{h'_0}{h_0}\psi'_{\eta} \\
&\quad +\bigg( c-\frac{1}{N-n}+\frac{2(\ez -1)}{n-1} \bigg) (\psi'_{\eta})^2 \\
&=  -\ez^2\frac{(n-1)(N-n)}{N-1} \frac{(h'_0)^2}{h_0^2}
 -2\ez \bigg( 1-\ez \frac{N-n}{N-1} \bigg) \frac{h'_0}{h_0}\psi'_{\eta} \\
&\quad -\bigg( \frac{N-1}{N-n} -2\ez +\frac{\ez^2 (N-n)}{N-1} \bigg)
 \frac{(\psi'_{\eta})^2}{n-1} \\
&= -\bigg( \ez \sqrt{\frac{(n-1)(N-n)}{N-1}} \frac{h'_0}{h_0}
 \pm \sqrt{\frac{N-1}{N-n} -2\ez +\frac{\ez^2 (N-n)}{N-1}} \frac{\psi'_{\eta}}{\sqrt{n-1}} \bigg)^2 \\
& \le 0,
\end{align*}
where we choose `$+$' if $1 -\ve(N-n)/(N-1) \ge 0$ and `$-$' otherwise.
Therefore we obtain
\[ h''_1(\tau) \le -c \e^{-\frac{4(\ez-1)}{n-1}\psi_{\eta}(\varphi_{\eta}^{-1}(\tau))}
 h_1(\tau) \Ric_N\Big( \dot{\eta}\big( \varphi_{\eta}^{-1}(\tau) \big) \Big)
 =-ch_1(\tau) \Ric_N\big( (\eta \circ \varphi_{\eta}^{-1})'(\tau) \big), \]
since
\begin{equation}\label{eq:eta'}
\dot{\eta}(t) =\e^{\frac{2(\ez -1)}{n-1}\psi_{\eta}(t)}
 \cdot (\eta \circ \varphi_{\eta}^{-1})'\big( \varphi_{\eta}(t) \big).
\end{equation}
This completes the proof for $N \in (-\infty,1) \cup (n,+\infty]$.
Then the cases of $N=1,n$ follow by taking the limits.
$\qedd$
\end{proof}

The \emph{diameter} of $(M,F)$ is defined by $\diam(M) :=\sup_{x,y \in M} d(x,y)$.
Along a geodesic $\eta:[0,l) \lra M$, we say that $\eta(t_0)$ is a \emph{conjugate point}
to $\eta(0)$ if there is a nontrivial Jacobi field $J$ vanishing at $0$ and $t_0$.
Equivalently, $\eta(t_0)$ is a conjugate point if $\td(\exp_{\eta(0)})(t_0 \dot{\eta}(0))$
does not have full rank.
In this case, $\eta$ is no more minimizing beyond $t_0$,
so that finding a conjugate point yields, by the Hopf--Rinow theorem,
a diameter bound (and singularity theorems in the Lorentzian setting).

\begin{theorem}[Bonnet--Myers Theorem]\label{th:F-BM}
Let $(M,F,\psi)$ be forward complete and $N \in (-\infty,1] \cup [n,+\infty]$,
$\ez$ in the $\ez$-range \eqref{eq:F-eran}, $K>0$ and $b>0$.
Assume that
\begin{equation}\label{eq:riclb}
\Ric_N(v) \ge KF^2(v) \e^{\frac{4(\ez-1)}{n-1}\psi(v)}
\end{equation}
holds for all $v \in TM \setminus 0$ and
\begin{equation}\label{eq:wb}
\e^{-\frac{2(\ez-1)}{n-1}\psi}\le b.
\end{equation}
Then we have
\[ \diam(M) \le \frac{b\pi}{\sqrt{cK}}. \]
In particular, $M$ is compact and has finite fundamental group.
\end{theorem}

We remark that, to be precise, the forward completeness is a condition on $(M,F)$
and the weight function $\psi$ plays a role in \eqref{eq:riclb} and \eqref{eq:wb}.

\begin{proof}
We will use the same notations as in Proposition~\ref{pr:F-Bish},
and show that any unit speed geodesic $\eta$ necessarily has a conjugate point
by the length $b\pi/\sqrt{cK}$.
By the Bishop inequality \eqref{eq:F-Bish} and the hypothesis \eqref{eq:riclb}
combined with \eqref{eq:eta'}, we have for positive $\tau$
\[ h''_1(\tau) \le -ch_1(\tau) K. \]

Now we shall prove that the limit $\lim_{\tau \to 0} \tau h_1'(\tau)$ exists and is nonpositive.
Here we present a simple argument based on the above Bishop inequality.
Moreover, in this paragraph we are going to consider general $K \in \R$,
for later reference to this proof in the proofs of the Laplacian and Bishop--Gromov comparison theorems.
We observe from the definition of $h_1$ that $h_1(\tau)=O(\tau^{c(n-1)})$ as $\tau \to 0$,
and $0<c(n-1) \le 1$.
Hence $\tau h_1(\tau)$ is differentiable at $0$, however,
we need to be careful because it does not necessarily imply that $\tau h_1(\tau)$ is $C^1$ at $0$.
By the continuity of $h_1$, for sufficiently small $\tau>0$, we have
$\vert h_1(\tau) \vert \le 1$ and in particular $h''_1(\tau) \le \vert cK\vert$.
Hence the function $\hat h(\tau):=h_1(\tau)-\frac{\vert cK\vert }{2} \tau^2$ is concave in $\tau$ near $\tau =0$.
Let $f(\tau): = \hat h(\tau)-\tau \hat h'(\tau)$ be the ordinate of the intersection
between the tangent to the graph of $\hat h$ at $(\tau, \hat h(\tau))$ and the vertical axis.
By the concavity of $\hat h$, $f$ is non-decreasing in $\tau>0$ and $f(\tau) \ge \hat h(0)=0$.
Therefore the limit $\lim_{\tau \to 0}f(\tau)$ exists and we obtain
\[
 \lim_{\tau \to 0} \tau h_1'(\tau) =\lim_{\tau \to 0} \tau \hat h'(\tau) =-\lim_{\tau \to 0} f(\tau) \le 0.
\]

Comparing $h_1$ with $\bs(\tau):=\sin(\sqrt{cK}\tau)$ which satisfies
$\bs''(\tau)+cK\bs(\tau)=0$, we find
\[
(h_1' \bs -h_1 \bs')' \le 0
\]
and, by $\lim_{\tau \to 0} \tau h_1'(\tau) \le 0$,
\[
\lim_{\tau \to 0} \big( h'_1(\tau) \bs(\tau) -h_1(\tau) \bs'(\tau) \big) \le 0. \]
This implies $h_1' \bs -h_1 \bs' \le 0$ and hence $h_1/\bs$ is non-increasing.
Then, since $\bs(\pi/\sqrt{cK})=0$,
$h_1(\tau_0)=0$ necessarily holds at some $\tau_0 \in (0,\pi/{\sqrt{cK}}]$, and
$\eta(t_0)$ with $t_0:=\varphi_\eta^{-1}(\tau_0)$ is a conjugate point to $x=\eta(0)$.
Noticing $\varphi_\eta(t_0) \ge t_0/b$ by the hypothesis \eqref{eq:wb},
we obtain $t_0 \le b\tau_0 \le b\pi/\sqrt{cK}$.
Since $\eta$ was an arbitrary unit speed geodesic and $(M,F)$ is forward complete,
we conclude that $\diam(M) \le  b\pi/\sqrt{cK}$.

The compactness of $M$ is an immediate consequence of the Hopf--Rinow theorem.
Since the universal cover $\widetilde{M}$ equipped with the lifted metric
and weight function again satisfies \eqref{eq:riclb} and \eqref{eq:wb},
$\widetilde{M}$ is compact and the fundamental group of $M$ is finite.
$\qedd$
\end{proof}

We stress that Theorem~\ref{th:F-BM} covers both the unweighted and weighted cases simultaneously.
On the one hand, in the unweighted case where $\psi \equiv 0$,
choosing $N=n$, $\ez=1$ and $b=1$ gives the classical (unweighted) Bonnet--Myers bound
$\diam(M) \le \pi\sqrt{(n-1)/K}$ under $\Ric \ge K$ by Auslander \cite{Au}.
On the other hand, when $N \in [n,+\infty)$ and $\ez=1$, we can again take $b=1$ and recover
the weighted Bonnet--Myers bound $\diam(M) \le \pi\sqrt{(N-1)/K}$ under $\Ric_N \ge K$ in \cite{Oint}.
We also remark that, in the remaining case of $N \in (-\infty,1] \cup \{+\infty\}$,
one cannot in general bound the diameter under the constant curvature bound $\Ric_N \ge K$
(see \cite{WY} for some examples).
Therefore, assuming the modified bound $\Ric_N \ge K \e^{\frac{4(\ez-1)}{n-1}\psi}$
with $|\ez|<1$ is essential.
Moreover, by virtue of the monotonicity \eqref{eq:mono},
one can easily construct an example satisfying \eqref{eq:riclb} for some $N \le 1$
but $\Ric_{\infty}(v)<0$ for some $v$.

\begin{remark}\label{rm:b}
In the above proof we found $\tau_0 =\varphi_{\eta}(t_0) \le \pi/\sqrt{cK}$, which means
\[  \int_0^{t_0} \e^{\frac{2(\ez -1)}{n-1}\psi_{\eta}(s)} \,\td s \le \frac{\pi}{\sqrt{cK}}, \]
without the need for the bound \eqref{eq:wb} on the weight function $\psi$.
This can be regarded as a diameter bound with respect to a deformed length,
studied with $\ez=(N-1)/(N-n)$ in \cite[Theorem~2.2]{WY} ($N=1$)
and \cite[Theorem~2.7]{KL} ($N<1$).
\end{remark}

As a corollary to the theorem and remark above,
we have the following compactness theorem without \eqref{eq:wb}
(see \cite[Corollary~2.3]{WY} and \cite[Corollary~2.8]{KL}).

\begin{corollary}\label{cr:F-BM}
Let $(M,F,\psi)$ be forward complete and $N \in (-\infty,1] \cup [n,+\infty]$,
$\ez$ in the $\ez$-range \eqref{eq:F-eran} and $K>0$.
If
\[ \Ric_N(v) \ge KF^2(v) \e^{\frac{4(\ez-1)}{n-1}\psi(v)} \]
holds for all $v \in TM \setminus 0$ and $(M,F,\psi)$ is forward $\ez$-complete,
then $M$ is compact.
\end{corollary}

\begin{proof}
It is sufficient to show that $M$ is forward bounded.
By way of contradiction, suppose that there are a point $x \in M$
and a sequence $\{y_k\}_{k \in \N}$ such that $d(x,y_k) \to \infty$.
Let $v_k \in U_xM$ be a unit vector such that $\eta_k(t):=\exp_x(t v_k)$ gives
a minimal geodesic from $x$ to $y_k$.
Taking a subsequence if necessary,
we can assume that $v_k$ converges to some unit vector $v \in U_xM$
and put $\eta(t):=\exp_x(tv)$.
Now, it follows from Remark~\ref{rm:b} that
\[ \int_0^{d(x,y_k)} \e^{\frac{2(\ez -1)}{n-1}\psi_{\eta_k}(s)} \,\td s \le \frac{\pi}{\sqrt{cK}}. \]
Letting $k \to \infty$ yields
\[ \int_0^{\infty} \e^{\frac{2(\ez -1)}{n-1}\psi_{\eta}(s)} \,\td s \le \frac{\pi}{\sqrt{cK}}, \]
which contradicts the $\ez$-completeness of $\eta$.
Therefore $M$ is forward bounded and hence compact by the Hopf--Rinow theorem.
$\qedd$
\end{proof}

\subsection{Laplacian comparison theorem}\label{ssc:F-Lcomp}

Next we deal with the Laplacian comparison theorem for the distance function $u(x)=d(z,x)$
from a fixed point $z \in M$.
We say that $x \in M$ is a \emph{cut point} to $z$ if there is a minimal geodesic
$\eta:[0,1] \lra M$ from $z$ to $x$ such that its extension $\bar{\eta}:[0,1+\ve] \lra M$
is not minimizing for any $\ve>0$ (in fact this holds for any minimal geodesic to a cut point).
The set of all cut points to $z$ is called the \emph{cut locus} of $z$ and denoted by $\Cut(z)$.

Note that $u$ is $C^{\infty}$ outside $\{z\} \cup \Cut(z)$,
and every integral curve of $\Grad u$ is a unit speed geodesic.
Let $\eta:[0,l) \lra M$ be a unit speed minimal geodesic emanating from $z$ without cut point,
then we define the \emph{$\psi$-Laplacian} of $u$ by
\begin{equation}\label{eq:F-wLap}
\Lap_{\psi} u \big( \eta(t) \big) :=\Lap u\big( \eta(t) \big) -\psi'_{\eta}(t).
\end{equation}
Generalizing $\bs$ in the proof of Theorem~\ref{th:F-BM},
we define the comparison function $\bs_{\kappa}$ as
\begin{equation}\label{eq:bs}
\bs_{\kappa}(t) := \begin{cases}
 \frac{1}{\sqrt{\kappa}} \sin(\sqrt{\kappa}t) & \kappa>0, \\
 t & \kappa=0, \\
 \frac{1}{\sqrt{-\kappa}} \sinh(\sqrt{-\kappa}t) & \kappa<0,
 \end{cases}
\end{equation}
where $t \in [0,\pi/\sqrt{\kappa}]$ for $\kappa>0$ and $t \in \R$ for $\kappa \le 0$.
Observe that $\bs_{\kappa}$ solves $\bs''_{\kappa} +\kappa \bs_{\kappa}=0$
with $\bs_{\kappa}(0)=0$ and $\bs'_{\kappa}(0)=1$.

\begin{theorem}[Laplacian comparison theorem]\label{th:F-Lcomp}
Let $(M,F,\psi)$ be forward complete and $N \in (-\infty,1] \cup [n,+\infty]$,
$\ez \in \R$ in the $\ez$-range \eqref{eq:F-eran}, $K \in \R$ and $b \ge a>0$.
Assume that
\[ \Ric_N(v)\ge KF^2(v) \e^{\frac{4(\ez-1)}{n-1}\psi(v)} \]
holds for all $v \in TM \setminus 0$ and
\begin{equation}\label{eq:wab}
a \le \e^{-\frac{2(\ez-1)}{n-1}\psi} \le b.
\end{equation}
Then, for any $z \in M$, the distance function $u(x):=d(z,x)$ satisfies
\[ \Lap_{\psi} u(x) \le \frac{1}{c\rho} \frac{\bs'_{cK}(u(x)/b)}{\bs_{cK}(u(x)/b)} \]
on $M \setminus (\{z\} \cup \Cut(z))$, where $\rho:=a$ if $\bs'_{cK}(u(x)/b) \ge 0$
and $\rho:=b$ if $\bs'_{cK}(u(x)/b)<0$.
\end{theorem}

Note that we have $\bs'_{cK}(u(x)/b)<0$ only when $K>0$ and $u(x) >b\pi/(2\sqrt{cK})$,
and in this case the assumption $\e^{-\frac{2(\ez-1)}{n-1}\psi} \ge a$ is unnecessary.
We also remark that, if $K>0$, then $u(x)/b<\pi/\sqrt{cK}$ thanks to Theorem~\ref{th:F-BM}
and the hypothesis $x \not\in \Cut(z)$.

\begin{proof}
We fix a unit tangent vector $v \in U_zM$, take the geodesic $\eta(t):=\exp_z(tv)$
and again make use of the same notations as in Subsection~\ref{ssc:F-BM}.
Let $l_v>0$ be the supremum of $t>0$ such that there is no cut point to $z$ on $\eta((0,t))$.
In the polar coordinates $(x^i)_{i=1}^n$ around $\eta((0,l_v))$ such that
$x^n=u$ and $(\del/\del x^i)|_{\eta(t)}=E_i(t)$, we shall first see that
\begin{equation}\label{eq:Lapu}
\Lap_{\psi} u\big( \eta(t) \big) =-\psi'_{\eta}(t)
 +\frac{\td}{\td t}\lf[ \log\lf( \sqrt{\det[g_{ij}(\dot{\eta})]} \r) \r],
\end{equation}
where one can take $\det[g_{ij}(\dot{\eta})]$ for $i,j=1,2,\ldots,n-1$
since $g_{in}(\dot{\eta})=0$ for $i=1,2,\ldots, n-1$ (by the Gauss lemma; see \cite[Lemma~6.1.1]{BCS})
and $g_{nn}(\dot{\eta})=1$.
By comparing \eqref{eq:Lapu} with the definition \eqref{eq:F-wLap} of $\Lap_{\psi}u(\eta(t))$,
it suffices to show that the second term in the right hand side of \eqref{eq:Lapu} coincides with
the unweighted Laplacian $\Lap u(\eta(t))$.
To this end, on the one hand, let us observe $\Grad u=\del/\del x^n$ and
\[ \Grad^2 u\bigg( \frac{\del}{\del x^i} \bigg)
 =D^{\Grad u}_{\del/\del x^i} \bigg( \frac{\del}{\del x^n} \bigg)
 =\sum_{j=1}^n \Gamma^j_{in}(\Grad u) \frac{\del}{\del x^j} \]
(we will suppress the evaluations at $\eta(t)$), and hence
\begin{align*}
\Lap u &=\trace(\Grad^2 u) =\sum_{i=1}^n \Gamma^i_{in}(\Grad u) \\
&=\frac{1}{2} \sum_{i,k=1}^n g^{ik}(\Grad u) \frac{\del g_{ik}}{\del x^n}(\Grad u)
 -\sum_{i,k,l=1}^n g^{ik}(\Grad u) C_{kil}(\Grad u) N_n^l(\Grad u) \\
&= \frac{1}{2} \sum_{i,k=1}^n g^{ik}(\Grad u) \frac{\del g_{ik}}{\del x^n}(\Grad u),
\end{align*}
where we used the geodesic equation \eqref{eq:geod} for $\eta$ to see
$N^l_n(\Grad u) =2G^l(\Grad u) =-\ddot{\eta}^l =0$.
On the other hand,
\[ \frac{\td}{\td t}\lf[ \log\lf( \sqrt{\det[g_{ij}(\dot{\eta})]} \r) \r]
 =\frac{1}{2} \trace\lf[ \bigg( \frac{\td[g_{ij}(\dot{\eta})]}{\td t} \bigg) \cdot \big( g^{jk}(\dot{\eta}) \big) \r]
 =\frac{1}{2} \sum_{i,j=1}^n \frac{\del g_{ij}}{\del x^n}(\dot{\eta}) g^{ji}(\dot{\eta}) \]
since $\ddot{\eta}^l =0$, thereby we obtain \eqref{eq:Lapu}.

Now, putting $h_0=(\det[g_{ij}(\dot{\eta})])^{1/2(n-1)}$ as in the proof of Proposition~\ref{pr:F-Bish},
we find that
\[ \Lap_{\psi} u\big( \eta(t) \big) =-\psi'_{\eta}(t) +\frac{(h_0^{n-1})'}{h_0^{n-1}}(t)
 =\frac{(\e^{-\psi_{\eta}} h_0^{n-1})'}{\e^{-\psi_{\eta}} h_0^{n-1}}(t). \]
Recall that
\[ (\e^{-\psi_{\eta}} h_0^{n-1})(t) =h(t)^{1/c} =h_1\big( \varphi_{\eta}(t) \big)^{1/c}, \]
and one can show that $h_1/\bs_{cK}$ is non-increasing in the same way as in Theorem~\ref{th:F-BM}.
Therefore $ (\e^{-\psi_{\eta}} h_0^{n-1})/\bs_{cK}(\varphi_{\eta})^{1/c}$ is non-increasing
and we have
\[ \frac{(\e^{-\psi_{\eta}} h_0^{n-1})'}{\e^{-\psi_{\eta}} h_0^{n-1}}(t)
 \le \frac{(\bs_{cK}(\varphi_{\eta})^{1/c})'}{\bs_{cK}(\varphi_{\eta})^{1/c}}(t)
 =\frac{1}{c} \frac{\bs'_{cK}(\varphi_{\eta}(t))}{\bs_{cK}(\varphi_{\eta}(t))} \varphi'_{\eta}(t)
 \le \frac{1}{c\rho} \frac{\bs'_{cK}(t/b)}{\bs_{cK}(t/b)} \]
by the fact that $\bs'_{\kappa}/\bs_{\kappa}$ is non-increasing for any $\kappa$
and by $b^{-1} \le \varphi'_{\eta} \le a^{-1}$ from \eqref{eq:wab}.
This completes the proof.
$\qedd$
\end{proof}

\begin{remark}\label{rm:ab}
The intermediate estimate
\[ \Lap_{\psi} u\big( \eta(t) \big) \le \e^{\frac{2(\ez -1)}{n-1}\psi_{\eta}(t)}
 \frac{\bs'_{cK}(\varphi_{\eta}(t))}{c \bs_{cK}(\varphi_{\eta}(t))} \]
(without the bound \eqref{eq:wab} on $\psi$)
in the above proof corresponds to \cite[Theorem~4.4]{WY} ($N=1$)
and \cite[Theorem~2.4]{KL} ($N<1$) for $\ez=(N-1)/(N-n)$ and $c=1/(n-N)$.
When $N \in [n,+\infty)$, $\ez=1$ and $c=1/(N-1)$,
we can take $a=b=1$ and recover \cite[Theorem~5.2]{OSheat}.
\end{remark}

We finally remark that, in the above proof,
we made use of the special property of the distance function $u$
that every integral curve of $\Grad u$ is a geodesic.
In dealing with more general functions, the usefulness of this type of Laplacian
(which is associated with a weight function $\psi$ not necessarily induced from a measure)
has yet to be shown.

\subsection{Bishop--Gromov comparison theorem}\label{ssc:F-BG}

We finally show the Bishop--Gromov volume comparison theorem,
for which we need a measure on $M$.
Let $\fm$ be a positive $C^{\infty}$-measure on $M$ and
$\psi_{\fm}$ be the weight function associated with $\fm$ (recall \eqref{eq:psi_m}).
We define the \emph{forward $r$-ball} of center $x$ as
\[ B^+(x,r):=\{ y \in M \,|\, d(x,y)<r \}. \]

\begin{theorem}[Bishop--Gromov comparison theorem]\label{th:F-BG}
Let $(M,F,\fm)$ be forward complete and $N \in (-\infty,1] \cup [n,+\infty]$,
$\ez \in \R$ in the $\ez$-range \eqref{eq:F-eran}, $K \in \R$ and $b \ge a>0$.
Assume that
\[ \Ric_N(v)\ge KF^2(v) \e^{\frac{4(\ez-1)}{n-1}\psi_{\fm}(v)} \]
holds for all $v \in TM \setminus 0$ and
\[ a \le \e^{-\frac{2(\ez-1)}{n-1}\psi_{\fm}} \le b. \]
Then we have
\[ \frac{\fm(B^+(x,R))}{\fm(B^+(x,r))}
\le \frac{b}{a}
 \frac{\int_0^{\min\{R/a,\,\pi/\sqrt{cK}\}} \bs_{cK}(\tau)^{1/c} \,\td\tau}{\int_0^{r/b} \bs_{cK}(\tau)^{1/c} \,\td\tau} \]
for all $x\in M$ and $0<r<R$, where $R \le b\pi/\sqrt{cK}$ when $K>0$
and we set $\pi/\sqrt{cK}:=\infty$ for $K \le 0$.
\end{theorem}

\begin{proof}
Given each unit vector $v \in U_xM$ and the geodesic $\eta(t):=\exp_x(tv)$,
$(h_1/\bs_{cK})^{1/c}$ is non-increasing as in the proof of Theorem~\ref{th:F-Lcomp}.
Hence the standard technique using Gromov's lemma (see \cite[Lemma~III.4.1]{Ch}) yields
that the integration is also non-increasing in the sense that
\[ \frac{\int_0^S h_1(\tau)^{1/c} \,\td \tau}{\int_0^S \bs_{cK}(\tau)^{1/c} \,\td \tau}
 \le \frac{\int_0^s h_1(\tau)^{1/c} \,\td \tau}{\int_0^s \bs_{cK}(\tau)^{1/c} \,\td \tau} \]
for $0<s<S$.
Observe from $b^{-1} \le \varphi'_{\eta} \le a^{-1}$ that
\[ \int_0^S h_1(\tau)^{1/c} \,\td \tau
 =\int_0^{\varphi_{\eta}^{-1}(S)} h(t)^{1/c} \varphi'_{\eta}(t) \,\td t
 \ge \frac{1}{b} \int_0^{\varphi_{\eta}^{-1}(S)} h(t)^{1/c} \,\td t \]
and
\[ \int_0^s h_1(\tau)^{1/c} \,\td \tau \le \frac{1}{a} \int_0^{\varphi_{\eta}^{-1}(s)} h(t)^{1/c} \,\td t. \]
Therefore we have
\[  \frac{\int_0^S h(t)^{1/c} \,\td t}{\int_0^s h(t)^{1/c} \,\td t}
 \le \frac{b}{a} \frac{\int_0^{\varphi_{\eta}(S)}
 h_1(\tau)^{1/c} \,\td \tau}{\int_0^{\varphi_{\eta}(s)} h_1(\tau)^{1/c} \,\td \tau}
 \le \frac{b}{a} \frac{\int_0^{\varphi_{\eta}(S)}
 \bs_{cK}(\tau)^{1/c} \,\td \tau}{\int_0^{\varphi_{\eta}(s)} \bs_{cK}(\tau)^{1/c} \,\td \tau}. \]

We shall integrate this inequality in $v \in U_xM$ with respect to the measure $\Xi$
induced from $g_v$.
For each $v \in U_xM$, let $l_v$ be the supremum of $t>0$ satisfying $d(x,\exp_x(tv))=t$.
Then we have, when $K>0$, $\varphi_{\eta}(l_v) \le \pi/\sqrt{cK}$ by the proof of Theorem~\ref{th:F-BM}
(recall Remark~\ref{rm:b}).
Moreover, $t/b \le \varphi_{\eta}(t) \le t/a$.
Therefore we obtain
\begin{align*}
\fm\big( B^+(x,R) \big)
&= \int_{U_xM} \int_0^{\min\{ R,\, l_v \}} h(t)^{1/c} \,\td t \,\Xi(\td v) \\
&\le \frac{b}{a}
 \frac{\int_0^{\min\{R/a,\,\pi/\sqrt{cK}\}} \bs_{cK}(\tau)^{1/c} \,\td \tau}{\int_0^{r/b} \bs_{cK}(\tau)^{1/c} \,\td \tau}
 \int_{U_xM} \int_0^{\min\{r,\, l_v\}} h(t)^{1/c} \,\td t \,\Xi(\td v) \\
&= \frac{b}{a}
 \frac{\int_0^{\min\{R/a,\,\pi/\sqrt{cK}\}} \bs_{cK}(\tau)^{1/c} \,\td \tau}{\int_0^{r/b} \bs_{cK}(\tau)^{1/c} \,\td \tau}
 \fm\big( B^+(x,r) \big)
\end{align*}
(notice that $r/b <\pi/\sqrt{cK}$ if $K>0$ by hypothesis).
This completes the proof.
$\qedd$
\end{proof}

This volume comparison theorem could be compared with \cite[Theorem~1.2]{WW}
on Riemannian manifolds $(M,g,\fm)$ with $\Ric_{\infty} \ge K$ and $|\psi_{\fm}| \le k$.
See also \cite[Theorem~4.5]{WY} and \cite[Theorem~2.10]{KL}
in terms of the deformed distance structure that we briefly discussed in Remark~\ref{rm:b}.

\section{Finsler spacetimes}\label{sc:LFmfd}

From here on we switch to the Lorentzian setting.
We refer to \cite{BEE,Min-Rev,ON} for the basics of Lorentzian geometry,
and to \cite{Min-spray,Min-causality} for further generalizations including Lorentz--Finsler manifolds
(see also Remark~\ref{rm:added} below).
In this and the next sections,
let $M$ be a connected $C^{\infty}$-manifold without boundary of dimension $n+1$.
We remark that $\dim M=n$ in the preceding sections, however,
it is standard in Lorentzian geometry to let $\dim M=n+1$,
we hope that this difference causes no confusion.
We will use indices in Greek: $\alpha,\beta=0,1,\ldots,n$.

\subsection{Lorentz--Finsler manifolds}\label{ssc:LFmfd}

Similarly to the preceding sections (and \cite{LMO}),
given local coordinates $(x^\alpha)_{\alpha=0}^n$ on an open set $U\subset M$,
we will use the coordinates
\[ v =\sum_{\beta=0}^n v^\beta \frac{\partial}{\partial x^\beta} \Big|_x, \qquad x \in U. \]
We follow Beem's definition \cite{Be} of a Finsler version of Lorentzian manifolds.

\begin{definition}[Lorentz--Finsler structures]\label{df:LFstr}
A \emph{Lorentz--Finsler structure} of $M$ will be a function
$L:TM \lra \R$ satisfying the following conditions:
\begin{enumerate}[(1)]
\item $L \in C^{\infty}(TM \setminus 0)$;
\item $L(cv)=c^2 L(v)$ for all $v \in TM$ and $c>0$;
\item For any $v \in TM \setminus 0$, the symmetric matrix
\[ \big( g_{\alpha \beta}(v) \big)_{\alpha,\beta=0}^n
 :=\bigg( \frac{\del^2 L}{\del v^\alpha \del v^\beta}(v) \bigg)_{\alpha,\beta=0}^n \]
is non-degenerate with signature $(-,+,\ldots,+)$.
\end{enumerate}
A pair $(M,L)$ is then called a ($C^{\infty}$-)\emph{Lorentz--Finsler manifold}.
\end{definition}

We say that $(M,L)$ is \emph{reversible} if $L(-v)=L(v)$ for all $v\in TM$.
For $v \in T_xM \setminus \{0\}$, define the Lorentzian metric $g_v$ of $T_xM$
in the same manner as \eqref{eq:gv} by
\[ g_v \Bigg( \sum_{\alpha=0}^n a_\alpha \frac{\del}{\del x^\alpha}\Big|_x,
 \sum_{\beta=0}^n b_\beta \frac{\del}{\del x^\beta}\Big|_x \Bigg)
 :=\sum_{\alpha,\beta=0}^n g_{\alpha \beta}(v) a_\alpha b_\beta. \]
Then we have $g_v(v,v)=2L(v)$.

\begin{definition}[Timelike vectors]\label{df:time}
A tangent vector $v \in TM $ is said to be \emph{timelike} (resp.\ \emph{null})
if $L(v)<0$ (resp.\ $L(v)=0$).
We say that $v$ is \emph{lightlike} if it is null and nonzero,
and \emph{causal} (or \emph{non-spacelike}) if it is timelike or lightlike ($L(v) \le 0$ and $v \neq 0$).
The \emph{spacelike vectors} are those for which $L(v)>0$ or $v=0$.
The set of timelike vectors will be denoted by
\[ \Omega'_x:=\{ v \in T_xM \,|\, L(v)<0 \}, \qquad \Omega' :=\bigcup_{x \in M} \Omega'_x. \]
\end{definition}

We will make use of the following function on $\Omega'$:
\begin{equation}\label{eq:LtoF}
F(v) :=\sqrt{-2L(v)} =\sqrt{-g_v(v,v)}.
\end{equation}
Note that $\Omega'_x \neq \emptyset$ and every connected component of $\Omega'_x$
is a convex cone (\cite{Be}, \cite[Lemma~2.3]{LMO}).
In general, the number of connected components of $\Omega_x'$ may be larger than 2
(see Example~\ref{ex:Beem}(b) below from \cite{Be}).
This fact will not affect our discussion because we shall deal with
only future-directed (timelike or causal) vectors; see Definition~\ref{df:spacetime} below.
We also remark that $\Omega'_x$ has exactly two connected components
in reversible Lorentz--Finsler manifolds of dimension $\ge 3$ (\cite[Theorem~7]{Min-cone}).

\begin{remark}\label{rm:added}
We comment on the differences in approach between
our Lorentz--Finsler setting and that adopted in some physical works.
Finslerian approaches to gravity have a venerable history,
one of the first formulations goes back to Horv\'ath \cite{horvath50} in the 1950s.
Since then many different Finslerian gravitational equations have appeared in the physical literature.
Due to the lack of exact solutions,
particularly of Finslerian generalizations of the Schwarzschild metric,
and of their confrontation with experiment,
a consensus on the correct Finslerian gravitational equation has not yet been reached.
Most equations (including Horv\'ath's) imply Ricci flatness in vacuum,
so this condition is often regarded as a minimal requirement.

Many authors worked via tensorial equations and
paid little attention on the constraints imposed by the Lorentzian signature
of the vertical Hessian of the Finsler Lagrangian.
Precisely, in physical papers, a direct product metric of the form
\[
2L\bigg( a\frac{\del}{\del t}+v \bigg) =-a^2 +F^2(v), \quad
 \bigg( a\frac{\del}{\del t},v \bigg) \in T\R \times T\Sigma,
\]
would be imposed as ansatz (see, e.g., \cite[(8)]{lammerzahl12}, \cite[(27)]{li14}, \cite[(34)]{rahaman15}).
These metrics are not of Lorentz--Finsler type according to our definition,
since the vertical Hessian at the observer (timelike vector) $\partial/\partial t$ is not well defined
($F^2$ is not twice differentiable at the origin; see, e.g., \cite[\S 1.2.2]{Obook}).
Less severe regularity problems are shared by those Lagrangians
that have no vertical Hessian at the light cone. 
This happens, for instance, to all the metrics that follow from the Bogoslovsky metric element
(very special relativity) \cite{fuster16,silagadze11}.
In these metrics lightlike particles might have infinite momenta,
making these models not as physically natural as one would desire.
The metrics of Lorentz--Randers type \cite{basilakos13,storer00,triantafyllopoulos20}
have a vertical Hessian that is also not Lorentzian and $C^2$ at the boundary of the light cone,
though the Finsler Lagrangian can be $C^1$.
Unfortunately, in these models the momenta of lightlike particles might vanish.

When it comes to work in Lorentz--Finsler geometry,
the ansatzes tried by physicists have the advantage of being simple,
and of making the calculations somewhat easier, but being non-$C^2$ at the light cone,
often produce, as shown previously, metrics that have undesired physical features.
From the point of view of pure mathematics,
it is not easy to work with these less regular models.
Nonetheless, some of our results could hold relaxing the $C^2$-regularity assumption at the light cone,
e.g., by imposing the $C^2$-Lorentzian condition only in the timelike cone.
This would raise issues related to the physical interpretation of light.
It has been proved in \cite[Section~4]{minguzzi15e} (and again in \cite[Theorem~6.6]{javaloyes18})
that lightlike geodesics and transport of momenta over them do not require
the vertical $C^2$-differentiability of the Finsler Lagrangian at the light cone,
and that these concepts follow just from the distribution of (anisotropic) cones,
not from the Finsler Lagrangian itself.

The problem of allowing for Lorentz--Finsler Langrangians with non-$C^2$ vertical behavior
at the light cone is somewhat analogous, in the positive signature,
to that of studying Finsler metrics $F$ for which the Hessian of $L=F^2/2$ is not well defined
or positive-definite in some directions (e.g., Kropina metrics).
Here one faces some annoying problems, for instance,
the Hopf--Rinow theorem does not hold \cite{sabau17}.
Our comparison results could be generalized to these frameworks,
as many of our proofs present arguments that seem localized over the indicatrix.
However, the impossibility of appealing to the Hopf--Rinow theorem
would certainly make such an investigation somewhat involved.
We do not attempt to generalize our results to that extent,
and keep the simpler framework of $C^2$-differentibility on the slit-tangent bundle.
This seems to be the right approach as our focus is on the role of the $\epsilon$-range concept
in comparison theorems rather than generality.
\end{remark}

\subsection{Causality theory}\label{ssc:causal}

Let $(M,L)$ be a Lorentz--Finsler manifold.

\begin{definition}[Finsler spacetimes]\label{df:spacetime}
If $(M,L)$ admits a timelike smooth vector field $X$ (namely $L(X(x))<0$ for all $x \in M$),
then $(M,L)$ is said to be \emph{time oriented} (by $X$).
A time oriented Lorentz--Finsler manifold will be called a \emph{Finsler spacetime}.
\end{definition}

In a Finsler spacetime oriented by $X$,
a causal vector $v \in T_xM$ is said to be \emph{future-directed}
if it lies in the same connected component of $\overline{\Omega'}\!_x \setminus \{0\}$ as $X(x)$.
We will denote by $\Omega_x \subset \Omega'_x$ the set of future-directed timelike vectors,
and define
\[ \Omega :=\bigcup_{x \in M} \Omega_x, \qquad
 \overline{\Omega} :=\bigcup_{x \in M} \overline{\Omega}_x, \qquad
 \overline{\Omega} \setminus 0 :=\bigcup_{x \in M} (\overline{\Omega}_x \setminus \{0\}). \]
A $C^1$-curve in $(M,L)$ is said to be \emph{timelike} (resp.\ \emph{causal})
if its tangent vector is always timelike (resp.\ causal).
All causal curves will be future-directed.

Given distinct points $x,y \in M$, we write $x \ll y$ (resp.\ $x<y$)
if there is a future-directed timelike (resp.\ causal) curve from $x$ to $y$,
and $x \le y$ means that $x=y$ or $x<y$.
Then we define the \emph{chronological past} and \emph{future} of $x$ by
\[ I^-(x):=\{y \in M \,|\, y \ll x\}, \qquad I^+(x):=\{y \in M \,|\, x \ll y\}, \]
and the \emph{causal past} and \emph{future} by
\[ J^-(x):=\{y \in M \,|\, y \le x\}, \qquad J^+(x):=\{y \in M \,|\, x \le y\}. \]
For a set $S\subset M$, we define $I^-(S),I^+(S),J^-(S)$ and $J^+(S)$ analogously.
Let us recall several causality conditions.

\begin{definition}[Causality conditions]\label{df:causal}
Let $(M,L)$ be a Finsler spacetime.
\begin{enumerate}[(1)]
\item $(M,L)$ is said to be \emph{chronological} if $x \notin I^+(x)$ for all $x\in M$.

\item We say that $(M,L)$ is \emph{causal} if there is no closed causal curve.

\item $(M,L)$ is said to be \emph{strongly causal} if, for all $x \in M$,
every neighborhood $U$ of $x$ contains another neighborhood $V$ of $x$
such that no causal curve intersects $V$ more than once.

\item We say that $(M,L)$ is \emph{globally hyperbolic}
if it is strongly causal and, for any $x,y \in M$, $J^+(x) \cap J^-(y)$ is compact (or empty).
\end{enumerate}
\end{definition}

It is straightforward that strong causality implies causality, and a causal spacetime is chronological.
A chronological spacetime is necessarily noncompact.

\subsection{Covariant derivative and Ricci curvature}\label{ssc:LFcurv}

One can introduce the covariant derivative and Ricci curvature
in the same way as in the positive-definite case.
We shall use the same notations as in Section~\ref{sc:Fmfd} and \cite{LMO}.

Similarly to Subsection~\ref{ssc:Fmfd}, we define
\[ \gamma^{\alpha}_{\beta \delta} (v)
 :=\frac{1}{2} \sum_{\lambda=0}^n g^{\alpha \lambda}(v)
 \bigg\{ \frac{\del g_{\lambda \delta}}{\del x^{\beta}}(v) +\frac{\del g_{\beta \lambda}}{\del x^{\delta}}(v)
 -\frac{\del g_{\beta \delta}}{\del x^{\lambda}}(v) \bigg\} \]
for $\alpha,\beta,\delta =0,1,\ldots,n$ and $v \in TM\setminus 0$,
where $(g^{\alpha \beta}(v))$ is the inverse matrix of $(g_{\alpha\beta}(v))$,
\[ G^{\alpha}(v) :=\frac{1}{2}\sum_{\beta,\delta=0}^n \gamma^{\alpha}_{\beta \delta}(v) v^{\beta} v^{\delta},
 \qquad N^{\alpha}_{\beta}(v) :=\frac{\del G^{\alpha}}{\del v^{\beta}}(v) \]
for $v \in TM \setminus 0$ ($G^{\alpha}(0)=N^{\alpha}_{\beta}(0):=0$ by convention), and
\[ \Gamma^{\alpha}_{\beta \delta}(v):=\gamma^{\alpha}_{\beta \delta}(v)
 -\frac{1}{2}\sum_{\lambda,\mu=0}^n g^{\alpha \lambda}(v)
 \bigg( \frac{\del g_{\lambda \delta}}{\del v^{\mu}}N^{\mu}_{\beta}
 +\frac{\del g_{\beta \lambda}}{\del v^{\mu}} N^{\mu}_{\delta}
 -\frac{\del g_{\beta \delta}}{\del v^{\mu}} N^{\mu}_{\lambda} \bigg)(v) \]
on $TM \setminus 0$.
Then the covariant derivative is defined in the same way as in Definition~\ref{df:covd},
\[ D_v^w X(x) :=\sum_{\alpha,\beta=0}^n
 \bigg\{ v^{\beta} \frac{\del X^{\alpha}}{\del x^{\beta}}(x)
 +\sum_{\delta=0}^n \Gamma^{\alpha}_{\beta \delta}(w) v^{\beta} X^{\delta}(x) \bigg\}
 \frac{\del}{\del x^{\alpha}} \Big|_x \in T_xM, \]
for a vector field $X$, $v \in T_xM$ and reference vector $w \in T_xM \setminus \{0\}$.

The geodesic equation for a causal curve $\eta:[0,1] \lra M$
is written as $D^{\dot{\eta}}_{\dot{\eta}}\dot{\eta} \equiv 0$ (recall \eqref{eq:geod}).
This is understood as the Euler--Lagrange equation associated with the action
\[ \mathcal{S}(\eta) :=\int_0^1 L\big( \dot{\eta}(t) \big) \,\td t. \]
The Lagrangian $L$ is preserved over a geodesic,
a fact which proves that the causal character of a geodesic is preserved,
hence we can speak of timelike and causal geodesics.

We also define the \emph{Lorentz--Finsler distance} $d(x,y)$ for $x,y \in M$ by
\[ d(x,y) :=\sup_{\eta} \int_0^1 F\big( \dot{\eta}(t) \big) \,\td t, \]
where $\eta:[0,1] \lra M$ runs over all causal curves from $x$ to $y$
(recall \eqref{eq:LtoF} for the definition of $F$).
We set $d(x,y):=0$ if there is no causal curve from $x$ to $y$ (namely $x \not< y$).
A constant speed causal curve attaining the above supremum, which is a causal geodesic,
is said to be \emph{maximal}.
In general, causal geodesics are locally maximizing much in the same way
as geodesics are locally minimizing in Riemannian geometry (\cite[Theorem~6]{Min-spray}).
The distance function $d$ is well-behaved in globally hyperbolic spacetimes as follows.

\begin{theorem}\label{th:connect}
If $(M,L)$ is globally hyperbolic, then the distance function $d$ is finite and continuous,
and any pair of points $x,y \in M$ with $x<y$ is connected by a maximal geodesic.
\end{theorem}

See \cite[Proposition~6.8]{Min-Ray} for the former claim.
The latter is the Finsler analogue of the \emph{Avez--Seifert theorem}
and found in \cite[Proposition~6.9]{Min-Ray}.
In general, $d$ is only lower semi-continuous (\cite[Proposition~6.7]{Min-Ray}) and can be infinite.

Next we introduce the Ricci curvature.
First of all, a $C^{\infty}$-vector field $J$ along a geodesic $\eta$ is called a \emph{Jacobi field}
if it is a solution to the equation
\[ D^{\dot{\eta}}_{\dot{\eta}} D^{\dot{\eta}}_{\dot{\eta}} J +R_{\dot{\eta}}(J) =0, \]
where
\[ R_v(w):=\sum_{\alpha,\beta=0}^n
 R^{\alpha}_{\beta}(v) w^{\beta} \frac{\del}{\del x^{\alpha}}\Big|_x \]
for $v,w \in T_xM$ and
\[ R^{\alpha}_{\beta}(v) :=2\frac{\del G^{\alpha}}{\del x^{\beta}}(v)
 -\sum_{\delta=0}^n \bigg\{ \frac{\del N^{\alpha}_{\beta}}{\del x^{\delta}}(v) v^{\delta}
 -2\frac{\del N^{\alpha}_{\beta}}{\del v^{\delta}}(v) G^{\delta}(v) \bigg\}
 -\sum_{\delta=0}^n N^{\alpha}_{\delta}(v) N^{\delta}_{\beta}(v) \]
is the \emph{curvature tensor}.
Similarly to Subsection~\ref{ssc:Fcurv}, a Jacobi field is also characterized as the variational vector field
of a geodesic variation.
Note that $R_v(w)$ is positively $2$-homogeneous in $v$ and linear in $w$.

\begin{definition}[Ricci curvature]\label{df:LF-curv}
For $v \in \overline{\Omega}_x$, we define the \emph{Ricci curvature}
(or \emph{Ricci scalar}) of $v$ as the trace of $R_v$: $\Ric(v):=\trace(R_v)$.
\end{definition}

We have $\Ric(cv)=c^2 \Ric(v)$ for $c>0$.
If $v$ is timelike, then one can also define the \emph{flag curvature}
\[ \bK(v,w) :=-\frac{g_v(R_v(w),w)}{g_v(v,v) g_v(w,w) -g_v(v,w)^2} \]
for $w \in T_xM$ linearly independent of $v$ (this is the opposite sign to \cite{BEE}),
and we have
\[ \Ric(v) =F^2(v) \sum_{i=1}^n \bK(v,e_i), \]
where $\{v/F(v)\} \cup \{e_i\}_{i=1}^n$ is an orthonormal basis of $(T_xM,g_v)$
(i.e., $g_v(e_i,e_j)=\delta_{ij}$ and $g_v(v,e_i)=0$ for all $i,j=1,2,\ldots,n$).
The Riemannian characterization of the Ricci (and flag) curvature
in the sense of Remark~\ref{rm:curv} is available also in this setting (see \cite[Theorem~3.7]{LMO}).

We summarize some basic properties of the curvature tensor
(see \cite[Proposition~2.4]{Min-Ray}).

\begin{lemma}\label{lm:LF-curv}
\begin{enumerate}[{\rm (i)}]
\item We have $R_v(v)=0$ for all $v \in \overline{\Omega}_x$.

\item  $g_v(v,R_v(w))=0$ for all $v \in \overline{\Omega}_x \setminus \{0\}$
and $w \in T_xM$.

\item $R_v$ is symmetric in the sense that $g_v(R_v(w_1),w_2)=g_v(w_1,R_v(w_2))$
for all $v \in \overline{\Omega}_x \setminus \{0\}$ and $w_1,\,w_2\in T_xM$.
\end{enumerate}
\end{lemma}

\subsection{Polar cones and Legendre transform}\label{ssc:dual}

In order to introduce the spacetime Laplacian (d'Alembertian),
we consider the dual structure to $L$ and the Legendre transform
(see \cite{Min-cone}, \cite[\S 3.1]{Min-causality} for further discussions).
Let $(M,L)$ be a Finsler spacetime.
Define the \emph{polar cone} to $\Omega_x$ by
\[ \Omega^*_x :=\big\{ \omega \in T_x^*M \,|\,
 \omega(v)<0\ \text{for all}\ v \in \overline{\Omega}_x \setminus \{0\} \big\}. \]
This is an open convex cone in $T_x^*M$.
For $\omega \in \Omega^*_x$, we define
\[ L^*(\omega) := -\frac{1}{2}\Big( \sup_{v \in \Omega_x \cap F^{-1}(1)} \omega(v) \Big)^2
 =-\frac{1}{2}\inf_{v \in \Omega_x \cap F^{-1}(1)} \big( \omega(v) \big)^2. \]
By definition, for any $v \in \Omega_x$ and $\omega \in \Omega^*_x$, we have
\[ L^*(\omega) \ge -\frac{1}{2}\bigg( \omega\bigg( \frac{v}{F(v)} \bigg) \bigg)^2
 =\frac{(\omega(v))^2}{4L(v)}. \]
This implies, since $L(v)<0$, the \emph{reverse Cauchy--Schwarz inequality}
\[ L^*(\omega) L(v) \le \frac{1}{4} \big( \omega(v) \big)^2 \]
(see also \cite[Theorem~3]{Min-cone}, \cite[Proposition~3.2]{Min-causality}).
Then we arrive at the following variational definition of the Legendre transform.

\begin{definition}[Legendre transform]\label{df:Leg}
Define the \emph{Legendre transform}
$\LL^*:\Omega^*_x \lra \Omega_x$ as the map sending $\omega \in \Omega^*_x$
to the unique element $v \in \Omega_x$ satisfying $L(v)=L^*(\omega)=\omega(v)/2$.
We also define $\LL^*(0):=0$.
\end{definition}

Note that the uniqueness of $v=\LL^*(\omega)$ follows from the strict convexity
of the super-level sets of $F$ in $\Omega_x$.
One can define $\LL:\Omega_x \lra \Omega^*_x$ in the same manner,
and then $\LL=(\LL^*)^{-1}$ holds by construction.
In order to write down $\LL^*$ and $\LL$ in coordinates, we introduce
\[ g^*_{\alpha \beta}(\omega)
 := \frac{\del^2 L^*}{\del \omega_{\alpha} \del \omega_{\beta}}(\omega) \]
for $\omega \in T^*M \setminus 0$.

\begin{lemma}[Coordinate expressions]\label{lm:Leg}
For $v \in \Omega_x$ and $\omega \in \Omega^*_x$,
we have in local coordinates around $x$
\begin{align*}
\LL(v) &= \sum_{\alpha=0}^n \frac{\del L}{\del v^{\alpha}}(v) \,\td x^{\alpha}
 =\sum_{\alpha,\beta=0}^n g_{\alpha \beta}(v) v^{\beta} \,\td x^{\alpha}, \\
\LL^*(\omega)
&= \sum_{\alpha=0}^n
 \frac{\del L^*}{\del \omega_{\alpha}}(\omega) \frac{\del}{\del x^{\alpha}}\Big|_x
 =\sum_{\alpha,\beta=0}^n
 g^*_{\alpha \beta}(\omega) \omega_{\beta} \frac{\del}{\del x^{\alpha}}\Big|_x.
\end{align*}
\end{lemma}

\begin{proof}
We consider only $\LL(v)$, the assertion for $\LL^*(\omega)$ is seen in the same way.
Fix $\bar{v} \in \Omega_x$ and put $\bar{\omega}:=\LL(\bar{v})$.
Then, by the definition of $L^*$, the function
$v \longmapsto \bar{\omega}(v)/\sqrt{-L(v)}$ on $\Omega_x$ attains its maximum at $v=\bar{v}$.
Hence we find
\[ \frac{\del}{\del v^{\alpha}} \bigg[ \frac{(\bar{\omega}(v))^2}{L(v)} \bigg]_{v=\bar{v}}
 =-\frac{1}{L^2(\bar{v})} \frac{\del L}{\del v^{\alpha}}(\bar{v}) \cdot \big( \bar{\omega}(\bar{v})\big)^2
 +\frac{2\bar{\omega}(\bar{v})}{L(\bar{v})} \bar{\omega}_{\alpha} =0 \]
for all $\alpha=0,1,\ldots,n$.
This implies, since $\bar{\omega}(\bar{v})=2L(\bar{v})$,
\[ \bar{\omega}_{\alpha}
 =\frac{1}{2} \frac{\bar{\omega}(\bar{v})}{L(\bar{v})} \frac{\del L}{\del v^{\alpha}}(\bar{v})
 =\frac{\del L}{\del v^{\alpha}}(\bar{v}). \]
This yields the first expression of $\LL(v)$,
and then the second is given by Euler's homogeneous function theorem.
$\qedd$
\end{proof}

Note that the expressions of $\LL$ and $\LL^*$ in the lemma
make sense for null and spacelike vectors as well.
Therefore we define
\[ \LL(v) := \sum_{\alpha=0}^n \frac{\del L}{\del v^{\alpha}}(v) \,\td x^{\alpha}, \qquad
 \LL^*(\omega)
 :=\sum_{\alpha=0}^n \frac{\del L^*}{\del \omega_{\alpha}}(\omega) \frac{\del}{\del x^{\alpha}} \]
for general $v \in TM$ and $\omega \in T^*M$ (one can readily see that they are well-defined).
This is indeed the usual definition of the Legendre transform,
and we summarize the basic properties in the next lemma
(see \cite[\S 2.4]{Min-cone} for further discussions).

\begin{lemma}[Properties of $\LL$ and $\LL^*$]\label{lm:Leg+}
\begin{enumerate}[{\rm (i)}]
\item
For any $x \in M$, $\LL$ is injective in each connected component of $\Omega'_x$.

\item
If $\dim M \ge 3$, then $\LL:T_xM \lra T_x^*M$ and $\LL^*:T_x^*M \lra T_xM$
are bijective at every $x \in M$.

\item
If $\dim M \ge 3$, then $\LL^*=\LL^{-1}$ holds on $T_x^*M$ and, for each $v \in \Omega_x$,
$(g^*_{\alpha \beta}(\LL(v)))$ is the inverse matrix of $(g_{\alpha \beta}(v))$.
\end{enumerate}
\end{lemma}

\begin{proof}
(i) and (ii) are proved by \cite[Theorem~5]{Min-cone} and \cite[Theorem~6]{Min-cone}, respectively.
Here we only show (iii) (see also \cite[Theorem~3.2]{Min-causality}).
By differentiating
\[ v =\LL^*\big( \LL(v) \big) =\sum_{\alpha=0}^n
 \frac{\del L^*}{\del \omega_{\alpha}}\big( \LL(v) \big) \frac{\del}{\del x^{\alpha}}\Big|_x \]
in $v^{\beta}$, we observe
\[ \delta_{\alpha \beta} =\sum_{\delta=0}^n
 \frac{\del^2 L^*}{\del \omega_{\delta} \del \omega_{\alpha}} \big( \LL(v) \big)
 \frac{\del^2 L}{\del v^{\beta} \del v^{\delta}}(v)
 =\sum_{\delta=0}^n g^*_{\alpha \delta}\big( \LL(v) \big) g_{\delta \beta}(v). \]
This completes the proof.
$\qedd$
\end{proof}

\begin{example}\label{ex:Beem}
\begin{enumerate}[(a)]
\item
In the standard Minkowski space $M=\R^{n+1}$ with
\[ L(v) =\frac{1}{2} \big\{ {-}(v^0)^2 +(v^1)^2 + \cdots +(v^n)^2 \big\}, \qquad
 \Omega_x =\big\{ (v^{\alpha})_{\alpha=0}^n \,\big|\, L(v)<0,\, v^0>0 \big\}, \]
in the canonical coordinates of $TM$ and $T^*M$, we have
\[ L^*(\omega) =\frac{1}{2} \big( {-}\omega_0^2 +\omega_1^2 + \cdots +\omega_n^2 \big), \qquad
 \Omega_x^* =\big\{ (\omega_{\alpha})_{\alpha=0}^n \,\big|\, L^*(\omega)<0,\, \omega_0<0 \big\}, \]
and $\LL(v)=(-v^0,v^1,\ldots,v^n)$.

\item
We shall see that the injectivity on the whole tangent space as in Lemma~\ref{lm:Leg+}(ii)
fails for $\dim M=2$.
Let us consider the Lorentz--Finsler structure
\[ L\bigg( r\cos\theta \frac{\del}{\del x}+r\sin\theta \frac{\del}{\del y} \bigg)
 :=\frac{1}{2} r^2 \cos k\theta \]
of $\R^2$ from \cite{Be} and \cite[Example~2.4]{LMO},
where $k \in \N$ and $(x,y)$ denotes the canonical coordinates
($k=2$ corresponds to the standard Minkowski space).
Note that, if we choose
\[ \Omega_x :=\bigg\{ r\cos\theta \frac{\del}{\del x} +r\sin\theta \frac{\del}{\del y} \,\bigg|\,
 r>0,\, \theta \in \bigg( \frac{\pi}{2k},\frac{3\pi}{2k} \bigg) \bigg\} \]
as future directions, then we have
\[ \Omega^*_x =\bigg\{ r\cos\theta \,\td x +r\sin\theta \,\td y \,\bigg|\,
 r>0,\, \theta \in \bigg( \frac{(3+k)\pi}{2k},\frac{(1+3k)\pi}{2k} \bigg) \bigg\}, \]
provided $k \ge 2$.
When $k=4$, one can rewrite $L$ as
\[ L\bigg( v \frac{\del}{\del x} +w \frac{\del}{\del y} \bigg)
 =\frac{(v^2 -w^2)^2 -(2vw)^2}{2(v^2 +w^2)} =\frac{v^4 -6v^2 w^2 +w^4}{2(v^2 +w^2)}, \]
and we observe from Lemma~\ref{lm:Leg} that
\[ \LL\bigg( v \frac{\del}{\del x} +w \frac{\del}{\del y} \bigg)
 =\bigg( v-\frac{8vw^4}{(v^2+w^2)^2} \bigg) \,\td x +\bigg( w-\frac{8v^4 w}{(v^2+w^2)^2} \bigg) \,\td y, \]
in other words,
\[ \LL\bigg( r\cos\theta \frac{\del}{\del x}+r\sin\theta \frac{\del}{\del y} \bigg)
 =r\cos\theta(1-8\sin^4 \theta) \,\td x +r\sin\theta(1-8\cos^4 \theta) \,\td y. \]
Therefore, for $\theta_1 \in (0,\pi/2)$ and $\theta_2 \in (\pi/2,\pi)$
with $\sin\theta_1 =\sin\theta_2 =8^{-1/4}$, we find
\[ \LL\bigg( r\cos\theta_1 \frac{\del}{\del x}+r\sin\theta_1 \frac{\del}{\del y} \bigg)
 =\LL\bigg( r\cos\theta_2 \frac{\del}{\del x}+r\sin\theta_2 \frac{\del}{\del y} \bigg). \]
\end{enumerate}
\end{example}

\subsection{Differential operators}\label{ssc:LF-Lap}

A continuous function $f:M \lra \R$ is called a \emph{time function}
if $f(x)<f(y)$ for all $x,y \in M$ with $x<y$.
A $C^1$-function $f:M \lra \R$ is said to be \emph{temporal} if $-\td f(x) \in \Omega^*_x$
for all $x \in M$.
Observe that temporal functions are time functions.

For a temporal function $f:M \lra \R$,
define the \emph{gradient vector} of $-f$ at $x \in M$ by
\[ \Grad(-f)(x) :=\LL^* \big( {-\td f}(x) \big) \in \Omega_x. \]
Note that, thanks to Lemmas~\ref{lm:Leg} and \ref{lm:Leg+}, we have for any $v \in T_xM$
\[ g_{\Grad(-f)} \big( \Grad(-f)(x),v \big)
 =-\sum_{\alpha,\beta,\delta=0}^n g_{\alpha \beta}\big( \Grad(-f)(x) \big)
 g^*_{\alpha \delta}\big( {-\td f(x)} \big) \frac{\del f}{\del x^{\delta}}(x) v^{\beta}
 =-\td f(v). \]
For a $C^2$-temporal function $f:M \lra \R$ and $x \in M$ (thereby $\Grad(-f)(x) \in \Omega_x$),
we define the \emph{Hessian} $\Grad^2 (-f):T_xM \lra T_xM$
in the same manner as \eqref{eq:F-Hess} by
\[ \Grad^2(-f)(v) :=D^{\Grad(-f)}_v \big( \Grad(-f) \big). \]
This spacetime Hessian has the same symmetry as in the positive-definite case,
let us give a proof (without coordinate calculations) for thoroughness.

\begin{lemma}[Symmetry of Hessian]\label{lm:Hess}
For a $C^2$-temporal function $f:M \lra \R$, we have
\[ g_{\Grad(-f)}\big( \Grad^2 (-f)(v),w \big) =g_{\Grad(-f)}\big( v,\Grad^2(-f)(w) \big) \]
for all $v,w \in T_xM$.
\end{lemma}

\begin{proof}
Put $h:=-f$ for brevity,
and let $V,W$ be extensions of $v,w$ to smooth vector fields around $x$, respectively.
Then we have
\begin{align*}
g_{\Grad h} \big( D^{\Grad h}_V(\Grad h),W \big)
&=V\big[ g_{\Grad h}(\Grad h,W) \big] -g_{\Grad h}(\Grad h,D^{\Grad h}_V W) \\
&=V[\td h(W)] -\td h(D^{\Grad h}_V W)
\end{align*}
(see \cite[Exercise~10.1.2]{BCS} for the first equality).
Combining this with $D^{\Grad h}_V W -D^{\Grad h}_W V =[V,W]$,
we obtain
\[ g_{\Grad h} \big( D^{\Grad h}_V(\Grad h),W \big) -g_{\Grad h} \big( D^{\Grad h}_W(\Grad h),V \big)
 =\td h([V,W]) -\td h([V,W]) =0 \]
as desired.
$\qedd$
\end{proof}

Similarly to \eqref{eq:F-Lap}, we define the spacetime \emph{Laplacian}
(or \emph{d'Alembertian}) as the trace of the Hessian,
\begin{equation}\label{eq:LF-Lap}
\Lap (-f):=\trace\big( \Grad^2(-f) \big),
\end{equation}
for $C^2$-temporal functions $f$.
We remark that this Laplacian is not elliptic but hyperbolic,
and is nonlinear (since the Legendre transform is nonlinear).

\section{Comparison theorems on weighted Finsler spacetimes}\label{sc:LFcomp}

Comparison theorems in Section~\ref{sc:Fcomp} can be generalized
to Finsler spacetimes in a suitable way.
We need to be careful with some Lorentzian behaviors and
introduce some special notions in Lorentzian geometry,
so we will give at least outlines of the proofs.
In addition, let us again stress that $\dim M=n+1$ (see also Remark~\ref{rm:N-n} below).

\subsection{Weighted Finsler spacetimes}\label{ssc:LF-wRic}

Let $(M,L)$ be a Finsler spacetime.
Similarly to Section~\ref{sc:Fcomp}, we employ a weight function
$\psi:\overline{\Omega} \setminus 0 \lra \R$ such that $\psi(cv)=\psi(v)$ for all $c>0$,
and set $\psi_{\eta}(t):=\psi(\dot{\eta}(t))$ along causal geodesics $\eta$
(as in \eqref{eq:psi_eta}).

\begin{definition}[Weighted Ricci curvature]\label{df:LF-wRic}
Given $v \in \overline{\Omega} \setminus 0$,
let $\eta:(-\ve,\ve) \lra M$ be the causal geodesic with $\dot{\eta}(0)=v$.
Then, for $N \in \R \setminus \{n\}$, define the \emph{weighted Ricci curvature} by
\[ \Ric_N(v) :=\Ric(v) +\psi''_{\eta}(0) -\frac{\psi'_{\eta}(0)^2}{N-n}. \]
We also define
\[ \Ric_{\infty}(v) :=\Ric(v) +\psi''_{\eta}(0),
 \qquad \Ric_n(v) :=\lim_{N \downarrow n} \Ric_N(v), \]
and $\Ric_N(0):=0$.
\end{definition}

\begin{remark}\label{rm:N-n}
Note that, despite $\dim M=n+1$,
the denominator $N-n$ in the last term of $\Ric_N$ is unchanged from \eqref{eq:F-wRic}.
Therefore $\Ric_N$ in the Lorentzian case corresponds to $\Ric_{N+1}$ in the positive-definite case.
In particular, $\Ric_0$ in this section corresponds to $\Ric_1$ in Section~\ref{sc:Fcomp}.
\end{remark}

We will say that $\Ric_N \ge K$ holds \emph{in timelike directions} for some $K \in \R$
if we have $\Ric_N(v) \ge KF^2(v) =-2KL(v)$ for all $v \in \Omega$
(recall \eqref{eq:LtoF} for the definition of $F$).

Due to our convention $\dim M=n+1$, we slightly modify the $\ez$-range
in Definition~\ref{df:F-eran} as follows (in the same form as in \cite{LMO}).

\begin{definition}[$\ez$-range]\label{df:LF-eran}
Given $N \in (-\infty,0] \cup [n,+\infty]$, we will consider $\ez \in \R$ in the following
\emph{$\ez$-range}:
\begin{equation}\label{eq:LF-eran}
\epsilon=0 \,\text{ for } N=0, \qquad
 \vert\epsilon\vert < \sqrt{\frac{N}{N-n}} \,\text{ for } N \neq 0,n, \qquad
 \ez \in \R \,\text{ for } N=n.
\end{equation}
The associated constant $c =c(N,\ez)$ is defined by
\begin{equation}\label{eq:LF-c}
c :=\frac{1}{n}\lf( 1-\ez^2\frac{N-n}{N} \r) >0
\end{equation}
for $N \ne 0$, and $c(0,0):=1/n$.
\end{definition}

Note that $\ez=1$ is admissible only for $N \in [n,+\infty)$, while $\ez=0$ is always admissible.
For a future-directed timelike geodesic $\eta:[0,l) \lra M$ and $\ez \in \R$, we set
\begin{equation}\label{eq:L-phi_e}
\varphi_{\eta}(t) := \int_0^t \e^{\frac{2(\ez-1)}{n} \psi_\eta(s)} \,\td s
\end{equation}
in the same way as \eqref{eq:phi_e} throughout this section.

\subsection{Bonnet--Myers theorem}\label{ssc:LF-BM}

We have shown in \cite[Theorem~5.17]{LMO}
the Bonnet--Myers theorem for weighted Finsler spacetimes
in the form that $\Ric_N \ge K>0$ with $N \in [n,+\infty)$ implies $\diam(M) \le \pi\sqrt{N/K}$
(we refer to \cite[Chapter~11]{BEE} for the Lorentzian case).
In order to generalize this to the one with $\ez$-range,
let us recall some notations and results of \cite{LMO}.

Given a timelike geodesic $\eta:[0,l) \lra M$ of \emph{unit speed} $F(\dot{\eta}) \equiv 1$
(equivalently, $L(\dot{\eta}) \equiv -1/2$),
we will denote by $N_\eta(t) \subset T_{\eta(t)}M$ the space of vectors
orthogonal to $\dot{\eta}(t)$ with respect to $g_{\dot{\eta}(t)}$.
For simplicity, the covariant derivative $D^{\dot{\eta}}_{\dot{\eta}}X$
of a vector field $X$ along $\eta$ will be denoted by $X'$.

\begin{definition}[Jacobi and Lagrange tensor fields]\label{df:Jtensor}
Let $\eta:[0,l) \lra M$ be a timelike geodesic of unit speed.
\begin{enumerate}[(1)]
\item
A smooth tensor field $\sJ$, giving an endomorphism
$\sJ(t):N_{\eta}(t) \lra N_{\eta}(t)$ for each $t \in [0,l)$,
is called a \emph{Jacobi tensor field along $\eta$} if we have
\begin{equation}\label{eq:Jtensor}
\sJ''+\sR \sJ=0
\end{equation}
and $\ker(\sJ(t)) \cap \ker(\sJ'(t)) =\{0\}$ for all $t$,
where $\sR(t):=R_{\dot\eta(t)}:N_{\eta}(t) \lra N_{\eta}(t)$ is the curvature endomorphism.

\item
A Jacobi tensor field $\sJ$ is called a \emph{Lagrange tensor field} if
\begin{equation}\label{eq:Ltensor}
(\sJ')^{\sT} \sJ -\sJ^{\sT} \sJ'=0
\end{equation}
holds on $[0,l)$, where the transpose $\sT$ is taken with respect to $g_{\dot \eta}$.
\end{enumerate}
\end{definition}

Some remarks on those notations are in order.

\begin{remark}\label{rm:Jtensor}
\begin{enumerate}[(a)]
\item
The equation \eqref{eq:Jtensor} means that, for any
$g_{\dot{\eta}}$-parallel vector field $P$ along $\eta$ (namely $P' \equiv 0$),
$Y(t):=\sJ(t)(P(t))$ is a Jacobi field along $\eta$.
Then the condition $\ker(\sJ(t)) \cap \ker(\sJ'(t)) =\{0\}$
implies that $Y=\sJ(P)$ is not identically zero for every nonzero $P$.
Note also that Lemma~\ref{lm:LF-curv}(ii) ensures
$R_{\dot{\eta}(t)}(w) \in N_{\eta}(t)$ for all $w \in T_{\eta(t)}M$.

\item
The equation \eqref{eq:Ltensor} means that
$\sJ^\sT\sJ'$ is $g_{\dot\eta}$-symmetric, precisely,
given two $g_{\dot{\eta}}$-parallel vector fields $P_1,P_2$ along $\eta$,
the Jacobi fields $Y_i:=\sJ(P_i)$ satisfy
\begin{equation}\label{eq:Ltensor'}
g_{\dot{\eta}}(Y'_1,Y_2) -g_{\dot{\eta}}(Y_1,Y'_2) \equiv 0.
\end{equation}
Since \eqref{eq:Jtensor} and Lemma~\ref{lm:LF-curv}(iii)
(with the help of \cite[(3.1)]{LMO}, see also \cite[Exercise~5.2.3]{BCS}) yield that
$[g_{\dot{\eta}}(Y'_1,Y_2) -g_{\dot{\eta}}(Y_1,Y'_2)]' \equiv 0$,
we have \eqref{eq:Ltensor'} for all $t$ if it holds at some $t$.
\end{enumerate}
\end{remark}

Given a Lagrange tensor field $\sJ$ along $\eta$, define $\sB:=\sJ' \sJ^{-1}$,
which is symmetric by \eqref{eq:Ltensor}.
We remark that $A$ (resp.\ $B,R$) in Section~\ref{sc:Fcomp} corresponds to
$\sJ^{\sT} \sJ$ (resp.\ $\sJ^{\sT} \sB (\sJ^{\sT})^{-1}, \sJ^{\sT} \sR \sJ$),
and that $A'=2BA$ in Lemma~\ref{lm:F-Bish} is equivalent to $\sB=\sJ' \sJ^{-1}$.
Multiplying \eqref{eq:Jtensor} by $\sJ^{-1}$ from right,
we arrive at the corresponding \emph{Riccati equation}
\[ \sB' +\sB^2 +\sR=0 \]
(see \cite[(5.3)]{LMO}, compare this with \eqref{eq:F-Ricc}).
We further define the \emph{expansion scalar}
\[ \theta(t) :=\trace\big( \sB(t) \big), \]
and  the \emph{shear tensor} (the traceless part of $\sB$)
\[ \sigma(t) :=\sB(t) -\frac{\theta(t)}{n} \sI_n(t), \]
where $\sI_n(t)$ denotes the identity of $N_{\eta}(t)$.

The weighted counterparts will make use of the parametrization $\varphi_{\eta}$ in \eqref{eq:L-phi_e}.
Note that, similarly to \eqref{eq:eta'},
\[ (\eta \circ \varphi_{\eta}^{-1})'(\tau)
 =\e^{-\frac{2(\ez-1)}{n}\psi_{\eta}(\varphi_{\eta}^{-1}(\tau))} \dot{\eta}\big( \varphi_{\eta}^{-1}(\tau) \big) \]
for $\tau \in [0,\varphi_{\eta}(l))$.
Define, for $\ez \in \R$ and $t \in [0,l)$,
\begin{align*}
\sJ_{\psi}(t) &:= \e^{-\psi_{\eta}(t)/n} \sJ(t),
\end{align*}
and for $t \in (0,l)$,
\begin{align*}
\sB_{\ez}(t) &:= (\sJ_{\psi} \circ \varphi_{\eta}^{-1})' \big( \varphi_{\eta}(t) \big) \cdot \sJ_{\psi}(t)^{-1}
 =\e^{-\frac{2(\ez-1)}{n}\psi_{\eta}(t)} \bigg( \sB(t)-\frac{\psi'_{\eta}(t)}{n}\sI_n(t) \bigg), \\
\theta_{\ez}(t) &:= \trace\big( \sB_{\ez}(t) \big)
 =\e^{-\frac{2(\ez-1)}{n}\psi_{\eta}(t)} \big( \theta(t)-\psi'_{\eta}(t) \big), \\
\sigma_{\ez}(t) &:= \sB_{\ez}(t) -\frac{\theta_{\ez}(t)}{n} \sI_n(t)
 =\e^{-\frac{2(\ez-1)}{n}\psi_{\eta}(t)} \sigma(t).
\end{align*}
Then the \emph{weighted Riccati equation} is given by
\[ (\sB_{\ez} \circ \varphi_{\eta}^{-1})'
 +\frac{2\ez}{n} (\psi_{\eta} \circ \varphi_{\eta}^{-1})' \cdot \sB_{\ez}(\varphi_{\eta}^{-1})
 +\sB_{\ez}^2(\varphi_{\eta}^{-1}) +\sR_{(0,\ez)}(\varphi_{\eta}^{-1}) =0 \]
on $(0,\varphi_{\eta}(l))$, where
\[ \sR_{(N,\ez)}(t) :=\e^{-\frac{4(\ez-1)}{n}\psi_{\eta}(t)}
 \bigg\{ \sR(t) +\frac{1}{n} \bigg( \psi''_{\eta}(t) -\frac{\psi'_{\eta}(t)^2}{N-n} \bigg) \sI_n(t) \bigg\} \]
(\cite[Lemma~5.5]{LMO}).
Observe that $\trace(\sR_{(N,\ez)}(t))=\Ric_N((\eta \circ \varphi_{\eta}^{-1})'(\varphi_{\eta}(t)))$.

We shall need the timelike weighted \emph{Raychaudhuri inequality},
which was proved in \cite[Proposition 5.7]{LMO} as a consequence of the above
weighted Riccati equation.

\begin{theorem}[Raychaudhuri inequality]\label{th:twri}
Let $\sJ$ be a nonsingular Lagrange tensor field along a timelike geodesic $\eta:[0,l) \lra M$ of unit speed.
Then, for every $\ez \in \R$ and $N \in (-\infty,0] \cup [n,+\infty]$, we have
\[ (\theta_\ez \circ \varphi_{\eta}^{-1})'
 \le -\Ric_N\! \big( (\eta \circ \varphi_{\eta}^{-1})' \big)
 -\trace\big( \sigma_{\ez}^2 (\varphi_{\eta}^{-1}) \big)
 -c\theta_{\ez}^2(\varphi_{\eta}^{-1}) \]
on $(0,\varphi_{\eta}(l))$ with $c=c(N,\ez)$ in \eqref{eq:LF-c}.
\end{theorem}

Now we can follow the lines of \cite[\S 5.5]{LMO} to see the Bonnet--Myers theorem
with $\ez$-range.
The \emph{timelike diameter} of $(M,L)$ is defined as $\diam(M):=\sup_{x,y \in M}d(x,y)$
(recall that $d(x,y)=0$ if $x \not< y$),
we refer to \cite[\S 11.1]{BEE} for some accounts on $\diam(M)$.
We remark that the finite diameter does not imply the compactness in the Lorentzian setting.

\begin{theorem}[Bonnet--Myers theorem]\label{th:LF-BM}
Let $(M,L,\psi)$ be a globally hyperbolic Finsler spacetime of dimension $n+1\ge 2$.
Suppose that, for some $N \in (-\infty,0] \cup [n,+\infty]$,
$\ez$ in the $\ez$-range \eqref{eq:LF-eran}, $K>0$ and $b>0$, we have
\begin{equation}\label{eq:LF-riclb}
\Ric_N(v)\ge KF^2(v) \e^{\frac{4(\ez-1)}{n}\psi(v)}
\end{equation}
for all $v \in \Omega$ and
\begin{equation}\label{eq:LF-wb}
\e^{-\frac{2(\ez-1)}{n}\psi} \le b.
\end{equation}
Then we have
\[ \diam(M) \le \frac{b\pi}{\sqrt{cK}}. \]
\end{theorem}

\begin{proof}
Suppose in contrary that there are $x,y\in M$ such that $l:=d(x,y)>b\pi/\sqrt{cK}$.
By Theorem~\ref{th:connect}, one can find a maximal timelike geodesic $\eta:[0,l] \lra M$
from $x$ to $y$ with $F(\dot{\eta}) \equiv 1$, and put $v:=\dot{\eta}(0) \in \Omega_x$.
Consider the Jacobi tensor field $\sJ$ given by
$\sJ(t)(w):=\td(\exp_x)_{tv}(tP(0))$ for $w \in N_{\eta}(t)$, where
$P$ is the $g_{\dot{\eta}}$-parallel vector field along $\eta$ with $P(t)=w$.
Then $\sJ$ is a Lagrange tensor field
(recall Remark~\ref{rm:Jtensor} and see the proof of \cite[Proposition~5.13]{LMO}).

Put
\[ h(t) := \big(\! \det \sJ_{\psi}(t) \big)^c =\e^{-c\psi_{\eta}(t)} \big(\! \det \sJ(t) \big)^c >0 \]
for $c$ in \eqref{eq:LF-c},
and $h_1(\tau):=h(\varphi_{\eta}^{-1}(\tau))$ for $\tau \in [0,\varphi_{\eta}(l))$
similarly to Proposition~\ref{pr:F-Bish}.
Then we have, since $\log h_1(\tau) =c \log[\det \sJ_{\psi}(\varphi_{\eta}^{-1}(\tau))]$,
\[ \frac{h'_1(\varphi_{\eta}(t))}{h_1(\varphi_{\eta}(t))} =c \trace\!\big( \sB_{\ez}(t) \big)
 =c\theta_{\ez}(t), \qquad
 \frac{h''_1 h_1 -(h'_1)^2}{h_1^2} =c(\theta_{\ez} \circ \varphi_{\eta}^{-1})'. \]
Hence it follows from Theorem~\ref{th:twri} that
\begin{equation}\label{eq:LF-Bish}
h''_1(\tau) \le -ch_1(\tau) \Ric_N\! \big( (\eta \circ \varphi_{\eta}^{-1})'(\tau) \big)
\end{equation}
for $\tau \in (0,\varphi_{\eta}(l))$ (as in \cite[Proposition~5.14]{LMO}).
This is exactly the analogue to the weighted Bishop inequality \eqref{eq:F-Bish}.
Under the hypotheses \eqref{eq:LF-riclb} and \eqref{eq:LF-wb},
we can show the existence of a conjugate point $\eta(t_0)$ to $\eta(0)$
for some $t_0 \le b\pi/\sqrt{cK}$ by the same argument as in Theorem~\ref{th:F-BM}.
This contradicts the maximality of $\eta$ and completes the proof.
$\qedd$
\end{proof}

Similarly to Remark~\ref{rm:b},
one can also obtain from the above proof the deformed diameter estimate
\[ \varphi_{\eta}(t_0) =\int_0^{t_0} \e^{\frac{2(\ez -1)}{n}\psi_{\eta}(s)} \,\td s
 \le \frac{\pi}{\sqrt{cK}} \]
without assuming \eqref{eq:LF-wb}.

\subsection{Laplacian comparison theorem}\label{ssc:LF-Lcomp}

Next we consider the Laplacian (d'Alembertian) comparison theorem with $\ez$-range,
as the Lorentzian counterpart to Theorem \ref{th:F-Lcomp}.
The Laplacian comparison theorem plays an essential role in the Lorentzian \emph{splitting theorem}
(see \cite[Chapter~14]{BEE}, \cite{Ca,WW2}).

Given $z \in M$, we say that $x \in I^+(z)$ is a \emph{timelike cut point} to $z$
if there is a maximal timelike geodesic $\eta:[0,1] \lra M$ from $z$ to $x$
such that its extension $\bar{\eta}:[0,1+\ve] \lra M$ is not maximal for any $\ve>0$.
The \emph{timelike cut locus} $\Cut(z)$ is the set of all cut points to $z$.
Notice that the function $u(x):=d(z,x)$ satisfies
$-\td u(x) \in \Omega_x^*$ for $x \in I^+(z) \setminus \Cut(z)$,
and hence $\Lap(-u)$ as in \eqref{eq:LF-Lap}
is well-defined on $I^+(z) \setminus \Cut(z)$.
Then, similarly to \eqref{eq:F-wLap}, we define the \emph{$\psi$-Laplacian} of $u$ (or $-u$) by
\[ \Lap_{\psi}(-u)(x) :=\Lap(-u)(x) -\psi'_{\eta}\big( d(z,x) \big) \]
on $I^+(z) \setminus \Cut(z)$, where $\eta:[0,d(z,x)] \lra M$
is the unique maximal timelike geodesic of unit speed from $z$ to $x$.
Recall \eqref{eq:bs} for the definition of $\bs_{\kappa}$.

\begin{theorem}[Laplacian comparison theorem]\label{th:LF-Lcomp}
Let $(M,L,\psi)$ be a globally hyperbolic Finsler spacetime of dimension $n+1 \ge 2$
and $N \in (-\infty,0] \cup [n,+\infty]$, $\ez \in \R$ in the $\ez$-range \eqref{eq:LF-eran},
$K \in \R$ and $b \ge a>0$.
Suppose that
\[ \Ric_N(v)\ge KF^2(v) \e^{\frac{4(\ez-1)}{n}\psi(v)} \]
holds for all $v \in \Omega$ and
\begin{equation}\label{eq:LF-wab}
a \le \e^{-\frac{2(\ez-1)}{n}\psi} \le b.
\end{equation}
Then, for any $z \in M$, the distance function $u(x):=d(z,x)$ satisfies
\[ \Lap_{\psi}(-u)(x) \le \frac{1}{c\rho} \frac{\bs'_{cK}(u(x)/b)}{\bs_{cK}(u(x)/b)} \]
on $I^+(z) \setminus \Cut(z)$, where $\rho:=a$ if $\bs'_{cK}(u(x)/b) \ge 0$
and $\rho:=b$ if $\bs'_{cK}(u(x)/b)<0$.
\end{theorem}

\begin{proof}
By the global hyperbolicity and $x \in I^+(z) \setminus \Cut(z)$,
there exists a unique maximal timelike geodesic $\eta(t)=\exp_z(tv)$ from $z$ to $x$
with $F(v)=1$.
Let $\sJ$ be the Lagrange tensor field along $\eta$ as in the proof of Theorem~\ref{th:LF-BM}.
Then the key ingredient of the proof is
\begin{equation}\label{eq:G^2=B}
\Grad^2(-u)|_{N_{\eta}(t)} =\sB(t)
\end{equation}
(which is a standard fact but we give a proof for completeness; see also \cite[Lemma~3.2]{OSbw}).
To this end, similarly to the proof of Theorem~\ref{th:F-Lcomp},
let $(x^{\alpha})_{\alpha=0}^n$ be polar coordinates around $\eta((0,d(z,x)))$
such that $x^0=u$ and $g_{\dot{\eta}}(\dot{\eta},\del/\del x^i)=0$ for all $i=1,2,\ldots,n$.
Note that $\Grad(-u)(\eta(t))=\dot{\eta}(t)=(\del/\del x^0)|_{\eta(t)}$.

Given $w \in N_{\eta}(t_0)$ with $t_0 \in (0,d(z,x))$,
let $P$ be the $g_{\dot{\eta}}$-parallel vector field along $\eta$ such that
$P(t_0)=\sJ(t_0)^{-1}(w)$.
Then, by the construction in the proof of Theorem~\ref{th:LF-BM}, we have
$w =\sJ(t_0)(P(t_0)) =\td(\exp_z)_{t_0 v}(t_0 P(0))$.
Put
\[ Y(t) :=\sJ(t) \big( P(t) \big) =\td(\exp_z)_{tv} \big( tP(0) \big)
 =\frac{\del}{\del \delta}\Big[ \exp_z \!\big( tv+\delta tP(0) \big) \Big] \Big|_{\delta=0}. \]
Let $Y(t)=\sum_{i=1}^n Y^i(t) (\del/\del x^i)|_{\eta(t)}$ and note that $(Y^i)' \equiv 0$
since we are considering the polar coordinates
(by exchanging the order of the derivatives in $\delta$ and $t$).
Hence, on the one hand, we have
\[ \sB(t_0)(w) =\sJ' \sJ^{-1}(w) =Y'(t_0)
 =\sum_{i,j=1}^n \Gamma^i_{j0} \big( \dot{\eta}(t_0) \big) Y^j(t_0) \frac{\del}{\del x^i}\Big|_x. \]
On the other hand,
\[ \Grad^2(-u)(w) =D^{\Grad(-u)}_w \big( \Grad(-u) \big)
 =\sum_{i,j=1}^n \Gamma^i_{j0}\big( \dot{\eta}(t_0) \big) w^j \frac{\del}{\del x^i}\Big|_x. \]
Since $Y(t_0)=w$, we obtain \eqref{eq:G^2=B}.

It follows from \eqref{eq:G^2=B} that
\begin{align*}
\Lap_{\psi}(-u)\big( \eta(t) \big)
&= \trace\big( \Grad^2(-u) \big) \big( \eta(t) \big) -\psi'_{\eta}(t)
 =\e^{\frac{2(\ez -1)}{n}\psi_{\eta}(t)} \trace\big( \sB_{\ez}(t) \big) \\
&= \e^{\frac{2(\ez -1)}{n}\psi_{\eta}(t)} \theta_{\ez}(t)
 =\e^{\frac{2(\ez -1)}{n}\psi_{\eta}(t)} \frac{h'_1(\varphi_{\eta}(t))}{ch_1(\varphi_{\eta}(t))},
\end{align*}
where the last equality was seen in the proof of Theorem~\ref{th:LF-BM}.
Combining this with $h'_1 \bs_{cK} -h_1 \bs'_{cK} \le 0$ shown in the same way as in
the proof of Theorem~\ref{th:F-BM} thanks to \eqref{eq:LF-Bish}, we have
\[ \Lap_{\psi}(-u)\big( \eta(t) \big)
 \le \e^{\frac{2(\ez -1)}{n}\psi_{\eta}(t)} \frac{\bs'_{cK}(\varphi_{\eta}(t))}{c\bs_{cK}(\varphi_{\eta}(t))}
 \le \frac{1}{c\rho} \frac{\bs'_{cK}(t/b)}{\bs_{cK}(t/b)} \]
by the fact that $\bs'_{cK}/\bs_{cK}$ is non-increasing
and by $b^{-1} \le \varphi'_{\eta} \le a^{-1}$ from \eqref{eq:LF-wab}.
This completes the proof.
$\qedd$
\end{proof}

Similarly to Remark~\ref{rm:ab}, the intermediate estimate
\[ \Lap_{\psi}(-u)\big( \eta(t) \big)
 \le \e^{\frac{2(\ez -1)}{n}\psi_{\eta}(t)} \frac{\bs'_{cK}(\varphi_{\eta}(t))}{c\bs_{cK}(\varphi_{\eta}(t))} \]
without the bound \eqref{eq:LF-wab} on $\psi$ is also meaningful.

\subsection{Bishop--Gromov comparison theorem}\label{ssc:LF-BG}

Volume comparison theorems in the Lorentzian setting are not as simple as in the positive-definite case.
This is because, given $x \in M$, the ``future ball'' $\{ y \in I^+(x) \,|\, d(x,y)<r \}$
is possibly noncompact and can have infinite volume.
For this reason, we need to restrict the directions to make the set of our interest be compact.
We shall make use of the following notion introduced in \cite{ES}.
We refer to \cite{EJK,Lu} for other volume comparison theorems in the same spirit,
the latter is concerned with weighted Finsler spacetimes.

\begin{definition}[SCLV]\label{df:SCLV}
For $x \in M$, a set $U \subset M$ is called
a \emph{standard for comparison of Lorentzian volumes} (\emph{SCLV} in short) at $x$
if there is $\wz U_x \subset T_xM$ satisfying the following conditions:
\begin{enumerate}[(1)]
\item $\wz U_x$ is an open set in $\Omega_x$; \label{sclv1}
\item $\wz U_x$ is star-shaped from the origin, i.e., we have $tv \in \wz U_x$
for all $v \in \wz U_x$ and $t \in (0,1)$; \label{sclv2}
\item $\wz U_x$ is contained in a compact set in $T_xM$; \label{sclv3}
\item The exponential map at $x$ is defined on $\wz U_x$,
the restriction of $\exp_x$ to $\wz U_x$ is a diffeomorphism
onto its image, and we have $U=\exp_x(\wz U_x)$. \label{sclv4}
\end{enumerate}
\end{definition}

Note that, for a small convex neighborhood $W$ of $0 \in T_xM$,
$\exp_x(W \cap \Omega_x)$ is an SLCV at $x$.
We need some more notation.
For $x,U,\wz U_x$ as above and $0<r \le 1$, we define
\[ \wz U_x(r):=\{rv \,|\, v \in \wz U_x\} \subset \wz U_x, \qquad
 U_x(r):=\exp_x\! \big( \wz U_x(r) \big) \subset U. \]
Since $U$ is not like a ``ball'' in general, we also define
\begin{align*}
\mathcal{U}_x &:=\{ v \in \Omega_x \,|\, F(v)=1,\, tv \in \wz U_x \text{ for some }t>0 \}, \\
T_{U,x}(v) &:= \sup\{ t>0 \,|\, tv \in \wz U_x \}, \quad v \in \mathcal{U}_x
\end{align*}
($T_{U,x}$ is called the \emph{cut function} in \cite{ES}).
Assuming that $T_{U,x}$ is constant on $\mathcal{U}_x$ amounts to considering (a part of) a ball.
Let $\fm$ be a positive $C^{\infty}$-measure on $M$
and $\psi_{\fm}$ be the weight function associated with $\fm$ in a similar way to \eqref{eq:psi_m},
precisely,
\[ \td\fm =\e^{-\psi_{\fm}(\dot{\eta}(t))} \sqrt{-\det\big[ g_{\alpha\beta}\big(\dot{\eta}(t) \big) \big]}
 \,\td x^0 \td x^1 \cdots \td x^n \]
along timelike geodesics $\eta$.

\begin{theorem}[Bishop--Gromov comparison theorem]\label{th:LF-BG}
Let $(M,L,\fm)$ be globally hyperbolic of dimension $n+1\ge 2$,
$N \in (-\infty,0] \cup [n,+\infty]$, $\ez \in \R$ in the $\ez$-range \eqref{eq:LF-eran},
$K \in \R$ and $b \ge a>0$.
Suppose that
\[ \Ric_N(v) \ge K F^2(v) \e^{\frac{4(\ez-1)}{n}\psi_{\fm}(v)} \]
holds for all $v \in \Omega$ and
\[ a \le \e^{-\frac{2(\ez-1)}{n}\psi_{\fm}} \le b. \]
Then, for any SCLV $U \subset M$ at $x \in M$ such that either
\begin{enumerate}[{\rm (A)}]
\item $T_{U,x} \equiv T$ on $\mathcal{U}_x$, or
\item $K=0$ and $T:=\inf_{v \in \mathcal{U}_x} T_{U,x}>0$,
\end{enumerate}
we have
\[ \frac{\fm(U_x(R))}{\fm(U_x(r))} \le \frac{b}{a}
 \frac{\int_0^{\min\{RT/a,\,\pi/\sqrt{cK}\}} \bs_{cK}(\tau)^{1/c} \,\td\tau}{\int_0^{rT/b} \bs_{cK}(\tau)^{1/c} \,\td\tau} \]
for all $0<r<R \le 1$, where we set $\pi/\sqrt{cK} :=\infty$ for $K \le 0$.
\end{theorem}

\begin{proof}
For each $v \in \mathcal{U}_x$ and the geodesic $\eta(t):=\exp_x(tv)$,
$h_1/\bs_{cK}$ is non-increasing as we mentioned in the proof of Theorem~\ref{th:LF-Lcomp}.
Hence we have
\[ \frac{\int_0^S h_1(\tau)^{1/c} \,\td \tau}{\int_0^S \bs_{cK}(\tau)^{1/c} \,\td \tau}
 \le \frac{\int_0^s h_1(\tau)^{1/c} \,\td \tau}{\int_0^s \bs_{cK}(\tau)^{1/c} \,\td \tau} \]
for $0<s<S$, similarly to the proof of Theorem~\ref{th:F-BG}.
Moreover, since $b^{-1} \le \varphi'_{\eta} \le a^{-1}$,
\[  \frac{\int_0^S h(t)^{1/c} \,\td t}{\int_0^s h(t)^{1/c} \,\td t}
 \le \frac{b}{a} \frac{\int_0^{\varphi_{\eta}(S)}
 h_1(\tau)^{1/c} \,\td \tau}{\int_0^{\varphi_{\eta}(s)} h_1(\tau)^{1/c} \,\td \tau}
 \le \frac{b}{a} \frac{\int_0^{\varphi_{\eta}(S)}
 \bs_{cK}(\tau)^{1/c} \,\td \tau}{\int_0^{\varphi_{\eta}(s)} \bs_{cK}(\tau)^{1/c} \,\td \tau}. \]
Now, letting $S=RT_{U,x}(v)$, $s=rT_{U,x}(v)$,
and noticing $\varphi_{\eta}(RT_{U,x}(v)) \le \pi/\sqrt{cK}$ if $K>0$
by the proof of Theorem~\ref{th:LF-BM} (or Theorem~\ref{th:F-BM}),
we deduce from the hypothesis (A) or (B) that (recall $\bs_0(\tau)=\tau$ from \eqref{eq:bs})
\begin{align*}
\frac{\int_0^{\varphi_{\eta}(RT_{U,x}(v))}
 \bs_{cK}(\tau)^{1/c} \,\td \tau}{\int_0^{\varphi_{\eta}(rT_{U,x}(v))} \bs_{cK}(\tau)^{1/c} \,\td \tau}
&\le \frac{\int_0^{\min\{ RT_{U,x}(v)/a,\,\pi/\sqrt{cK} \}} \bs_{cK}(\tau)^{1/c}
 \,\td \tau}{\int_0^{rT_{U,x}(v)/b} \bs_{cK}(\tau)^{1/c} \,\td \tau} \\
&\le \frac{\int_0^{\min\{ RT/a,\,\pi/\sqrt{cK} \}}
 \bs_{cK}(\tau)^{1/c} \,\td \tau}{\int_0^{rT/b} \bs_{cK}(\tau)^{1/c} \,\td \tau}.
\end{align*}
We integrate this inequality in $v \in \mathcal{U}_x$ with respect to the measure $\Xi$
induced from $g_v$ to see
\begin{align*}
\fm\big( U_x(R) \big)
&= \int_{\mathcal{U}_x} \int_0^{RT_{U,x}(v)} h(t)^{1/c} \,\td t \,\Xi(\td v) \\
&\le \frac{b}{a}
 \frac{\int_0^{\min\{ RT/a,\,\pi/\sqrt{cK} \}} \bs_{cK}(\tau)^{1/c} \,\td \tau}{\int_0^{rT/b} \bs_{cK}(\tau)^{1/c} \,\td \tau}
 \int_{\mathcal{U}_x} \int_0^{rT_{U,x}(v)} h(t)^{1/c} \,\td t \,\Xi(\td v) \\
&=  \frac{b}{a}
 \frac{\int_0^{\min\{ RT/a,\,\pi/\sqrt{cK} \}} \bs_{cK}(\tau)^{1/c} \,\td \tau}{\int_0^{rT/b} \bs_{cK}(\tau)^{1/c} \,\td \tau}
 \fm\big( U_x(r) \big)
\end{align*}
(we remark that $rT/b \le \varphi_{\eta}(rT) \le \pi/\sqrt{cK}$ if $K>0$).
This completes the proof.
$\qedd$
\end{proof}

\noindent
\textbf{Acknowledgements.}
EM thanks Department of Mathematics of Osaka University for kind hospitality.
SO was supported in part by JSPS Grant-in-Aid for Scientific Research (KAKENHI) 19H01786.

{\small

}

\end{document}